\documentclass[11pt]{amsart}

\usepackage{amssymb}

\input xy

\xyoption{all}

\theoremstyle{plain}

\newtheorem{theorem}{Theorem}[section]

\newtheorem{corollary}[theorem]{Corollary}

\newtheorem{subclaim}[theorem]{Subclaim}

\newtheorem{conjecture}[theorem]{Conjecture}

\newtheorem{proposition}[theorem]{Proposition}

\newtheorem{fact}[theorem]{Fact}

\newtheorem{claim}[theorem]{Claim}

\newtheorem{question}[theorem]{Question}

\theoremstyle{definition} %actually mainly don't "mute" (in hebrew)

\newtheorem{definition}[theorem]{Definition}

\newtheorem{example}[theorem]{Example}

\newtheorem{remark}[theorem]{Remark}

\theoremstyle{remark}

\newtheorem{statement}[theorem]{Statement}

\newtheorem{hypothesis}[theorem]{Hypothesis}

\newtheorem{axioms}[theorem]{Axioms}

\renewcommand{\phi}{\varphi}

\newcommand{\initial}\lessdot

\def\?{?\vadjust

{\vbox to 0pt{\vskip-7pt\hbox to 1.1\hsize{\hfill\huge ?!}}}}

\newcommand{\be}{\begin{enumerate}}
\newcommand{\ee}{\end{enumerate}}

\usepackage{enumerate}
\usepackage{verbatim}
\renewcommand{\epsilon}{\varepsilon}

 \def\nfork{\setbox0\hbox{$\bigcup$}%
 \setbox1=\hbox to \wd0{\hfil\vrule width 0.7pt depth 2pt height 7.5pt\hfil}%
 \wd1=0cm\relax\box1\box0}
 \def\dnf{\mathop{\nfork}\limits}

\begin{document}
%%%%%%%%%%%%%%%%%%%%%%%%%%%%%%%%
%%%%%%%%%%%%%%%%%%%%%%%%%%%%%%%%
\title{Tameness, Uniqueness Triples and Amalgamation}
\author{Adi Jarden}
\email[Adi Jarden]{jardena@ariel.ac.il}
\address{Department of Computer Science and Mathematics \\ Ariel University \\ Ariel, Israel}

\begin{abstract}
We combine two approaches to the study of classification theory of AECs:
\begin{enumerate}
\item that of Shelah: studying non-forking frames without assuming the amalgamation property but assuming the existence of uniqueness triples and \item that of Grossberg and VanDieren \cite{grva25}: (studying non-splitting) assuming the amalgamation property and tameness.
\end{enumerate}

In \cite{jrsh875} we derive a good non-forking $\lambda^+$-frame from a semi-good non-forking $\lambda$-frame. But the classes $K_{\lambda^+}$ and $\preceq \restriction K_{\lambda^+}$ are replaced: $K_{\lambda^+}$ is restricted to the saturated models and the partial order $\preceq \restriction K_{\lambda^+}$ is restricted to the partial order $\preceq^{NF}_{\lambda^+}$. 

Here, we avoid the restriction of the partial order $\preceq \restriction K_{\lambda^+}$, assuming that every saturated model (in $\lambda^+$ over $\lambda$) is an amalgamation base and $(\lambda,\lambda^+)$-tameness for non-forking types over saturated models, (in addition to the hypotheses of \cite{jrsh875}): Theorem \ref{the main theorem of the paper} states that $M \preceq M^+$ if and only if $M \preceq^{NF}_{\lambda^+}M^+$, provided that $M$ and $M^+$ are saturated models.

We present sufficient conditions for three good non-forking $\lambda^+$-frames: one relates to all the models of cardinality $\lambda^+$ and the two others relate to the saturated models only. By an `unproven claim' of Shelah, if we can repeat this procedure $\omega$ times, namely, `derive' good non-forking $\lambda^{+n}$ frame for each $n<\omega$ then the categoricity conjecture holds.

In \cite{vas6}, Vasey applies Theorem \ref{we can use NF}, proving the categoricity conjecture under the above `unproven claim' of Shelah. 
 
In \cite{jrprime}, we apply Theorem \ref{the main theorem of the paper}, proving the existence of primeness triples. 

\end{abstract}

\today 

\maketitle

\tableofcontents

\maketitle

\tableofcontents

\section{Introduction}
%This paper assumes that the reader is familiar with AECs, semi-good non-forking frames, uniqueness triples and %the existence property. (These topics are respectively discussed in Sections 2 and 4 and in Definition 3.2.1(2) of \cite{jrsh875}).

The notion of a good non-forking $\lambda$-frame was introduced by Shelah \cite[II]{shh}. It is an axiomatization of the non-forking relation in superstable first order theories. The goal of the study of good non-forking frames is to classify AECs. If the amalgamation property does not hold then the definition of a galois-type is problematic. So Shelah added the amalgamation property to the axioms of a good non-forking frame. 

Shelah \cite[II.3]{shh} found cases, where we can prove the amalgamation property in a specific cardinality $\lambda$ and to prove the existence of a non-forking relation, relating to models of cardinality $\lambda$. This is the reason, why Shelah defined the non-forking relation in a good non-forking frame as relating to models of a specific cardinality, $\lambda$, only (so the amalgamation in $\lambda$ property is one of the axioms of a good non-forking $\lambda$-frame, but the amalgamation property is not!). 

Shelah \cite[II]{shh} presented a way to extend a good non-forking $\lambda$-frame to models of cardinality greater than $\lambda$ and proved that several axioms are preserved. But the amalgamation property and Axioms \ref{4 axioms} are hard to be proved even for models of cardinality $\lambda^+$.
\begin{axioms}\label{4 axioms}
\mbox{}
\begin{enumerate}
\item Extension, \item Uniqueness, \item Basic stability and \item Symmetry.
\end{enumerate}
\end{axioms}

We now consider models of cardinality $\lambda^+$ only. In order to get the amalgamation property and Axioms \ref{4 axioms}, there were introduced two approaches:
\begin{enumerate}
\item Shelah's approach: to change the AEC, such that the amalgamation and Axioms \ref{4 axioms} will be satisfied, \item the tameness approach for non-forking frames: to add the tameness property to the hypotheses.
\end{enumerate}

In Shelah's approach, the relation $\preceq \restriction K_{\lambda^+}$ is restricted to the relation $\preceq^{NF}_{\lambda^+}$ (see Definition \ref{definition of preceq^{NF}}). One advantage of the relation $\preceq^{NF}_{\lambda^+}$ is that $(K_{\lambda^+},\preceq^{NF}_{\lambda^+})$ satisfies the amalgamation property (even if $(K_{\lambda^+},\preceq \restriction K_{\lambda^+})$ does not satisfy the amalgamation property). So we get artificially the amalgamation in $\lambda^+$ property. But a new problem arises: the pair $(K_{\lambda^+},\preceq^{NF}_{\lambda^+})$ may not satisfy smoothness (one of the axioms of AEC). In order to solve this problem, the class of models of cardinality $\lambda^+$ is restricted to the saturated models of cardinality $\lambda^+$ over $\lambda$ (and we assume that there are not many models of cardinality $\lambda^{++}$).  

Shelah \cite{shh}.II derived a good non-forking $\lambda^+$-frame, using Shelah's approach: he proved that in the new AEC (the class of saturated models with the relation $\preceq^{NF}_{\lambda^+}$) all the axioms of a good non-forking $\lambda^+$-frame are satisfied, assuming additional hypotheses. Jarden and Shelah \cite[Theorem 11.1.5]{jrsh875} generalized the work done in \cite{shh}.II: they introduced the notion of a semi-good non-forking $\lambda$-frame. It is a generalization of a good non-forking $\lambda$-frame, where the stability hypothesis is weakened. Jarden and Shelah proved that we can derive a good non-forking $\lambda^+$-frame, from a semi-good non-forking $\lambda$-frame, assuming similar additional hypotheses. 

In order to clarify the importance of Shelah's approach to the solution of the categoricity conjecture, we have to recall the following definition:
Roughly, we say that a good non-forking $\lambda$ frame is $n$-successful when we can derive a good non-forking $\lambda^{+m}$-frame for each $m \leq n$ (for a precise definition, see \cite[Definition 10.1.1]{jrsh875}). $\omega$-successful means $n$-successful for every $n<\omega$.

Shelah \cite[III.12.40]{shh} claims the following (he did not publish a proof yet):
\begin{conjecture}
Assume that $2^\lambda<2^{\lambda^+}$ for each cardinal $\lambda$. Let $(K,\preceq)$ be an AEC such that there is an $\omega$-successful good non-forking $\lambda$-frame with underlying class $K_\lambda$. Then $K$ is categorical in some $\mu>\lambda^{+\omega}$ if and only if $K$ is categorical in each $\mu>\lambda^{+\omega}$.
\end{conjecture}

The main advantage of Shelah's approach is that we do not assume that the amalgamation property holds.

Shelah's approach has two disadvantages:
\begin{enumerate}
\item We change the class $K$ and the partial order $\preceq$ and \item not every good non-forking $\lambda$-frame is $1$-successful.
\end{enumerate}

We now consider the tameness approach. Grossberg and VanDieren \cite{grva25} introduced the tameness property. They used this notion to prove an upward categoricity theorem. 

The assumption of tameness is reasonable, not only because the natural AECs are tame, but also because Boney and Unger proved that there is a connection between tameness and large cardinal axioms.

The following fact is due to Boney:
\begin{fact}\cite[Theorem 1.3]{bo2}:
If $(K,\preceq)$ is an AEC with $LST$-number less than $\kappa$ and $\kappa$ is strongly compact then $(K,\preceq)$ is $\kappa$-tame.
\end{fact}

Boney and Unger \cite[Corollary 4.13]{bo13un} proved the following:
\begin{fact}
The following are equivalent:
\begin{enumerate}
\item Every AEC is tame. \item There are class many almost strongly compact cardinals.
\end{enumerate}
\end{fact}

In 2006, \footnote{in a private conversation, during the American Institute of Mathematics Workshop on the Classification Theory for
Abstract Elementary Classes on Palo Alto} Grossberg raised the conjecture that one main values of tameness allows one to go from a non-forking $\lambda$-frame to a non-forking $\lambda^+$-frame without changing the AEC (in contrast to Shelah's approach). Several years ago, we checked the conjecture and saw that only the symmetry axiom does not hold by the known proofs.\footnote{On June 18 2013, we wrote in an e.mail to Grossberg, that we have almost proved the following thing: if we have a good non-forking $\lambda$-frame and the amalgamation and tameness properties hold for the eight successive cardinals then we have a good non-forking frames in these eight cardinals.}
\begin{comment}
to myself: I sent it to the second e.mail address of Rami. But in 2011 I told this idea to Will, when we met at the logic colloquium in Barcelone. I told it to Will again in 2013 when we met at the logic colloquium in Evora.
\end{comment}

This was the beginning of the tameness approach for non-forking frames. It is a combination of Grossberg and VanDieren's approach and Shelah's approach. 

Let us explain the idea of this conjecture. Since we do not want to change the AEC, we must assume the amalgamation in $\lambda^+$  property. Uniqueness (of the non-forking extension of a type) for models of cardinality $\lambda^+$ is implied easily by $(\lambda,\lambda^+)$-tameness (relating to basic types, see below). The proofs of the extension property and basic stability in $\lambda^+$ in \cite{jrsh875} can be applied here. So our main challenge is to get symmetry. We conjectured that symmetry holds as well.

Recently, the symmetry axiom was proved under three different hypotheses (in chronological order):
\begin{enumerate}
\item Boney \cite{bo3} proved it, assuming a strong version of tameness: tameness for two elements, \item in Proposition \ref{continuity implies symmetry}, it is proved, assuming the $(\lambda,\lambda^+)$-continuity of independence of sequences of length $2$ property and \item Boney and Vasey \cite{bo6va4} proved that tameness implies tameness for two elements, so actually they proved the symmetry axiom assuming tameness, solving our conjecture.
\end{enumerate}

[While VanDieren and Vasey \cite{vavas10} proved a downward transfer of the symmetry propertry for splitting, here the issue is an \emph{upward} transfer of the symmetry property of a \emph{general} non-forking notion].

%Let us sum up the main advantages and disanvantages of the tameness approach for non-forking frames:
%\begin{itemize}
%\item The main advantage: We avoid changing the AEC, in contrast to the artificial approach.  
%\item The main disadvantage: we assume the amalgamation property and tameness, as in the tameness approach.  
%\end{itemize}

It is desirable to avoid assuming the amalgamation in $\lambda^+$ property, whenever it is possible. So we first derive a good non-forking $\lambda^+$-frame minus amalgamation, assuming the $(\lambda,\lambda^+)$-continuity of serial independence property (`minus amalgamation' means that the amalgamation property in $\lambda^+$ may hold and may not hold). Then we present sufficient conditions for the $(\lambda,\lambda^+)$-continuity of serial independence property.

We prefer to assume that every saturated model is an amalgamation base than to assume that $(K,\preceq)$ satisfies the amalgamation property. This is the motivation for defining the second candidate of a good non-forking $\lambda^+$-frame: $\frak{s}^{sat}$. The AEC of $\frak{s}^{sat}$ is the class of saturated models in $\lambda^+$ over $\lambda$ with the relation $\preceq$. 

In Section 10, we recall the definition of the third candidate, $\frak{s^+}$, for a good non-forking $\lambda^+$-frame. The AEC of $\frak{s^+}$ is the class of saturated models in $\lambda^+$ over $\lambda$ with the relation $\preceq^{NF}_{\lambda^+}$.

The theorems in this paper are divided into two kinds:
\begin{enumerate}
\item Sufficient conditions for the equivalence between the relations $\preceq \restriction K_{\lambda^+}$ and $\preceq^{NF}_{\lambda^+}$ and \item sufficient conditions, under which $\frak{s}_{\lambda^+}$, $\frak{s}^{sat}$ or $\frak{s^+}$ is a good non-forking $\lambda^+$-frame.
\end{enumerate} 

Recall \cite[Definition 1.0.25]{jrsh875}:
\begin{definition}
$K^{sat}$ is the class of saturated models in $\lambda^+$ over $\lambda$.
\end{definition}

The main theorems of the paper are:
\begin{theorem}[Theorem \ref{the main theorem of the paper}]\label{0}
Suppose:
\begin{enumerate}
\item $K$ is categorical in $\lambda$, \item $\frak{s}$ is a semi-good non-forking $\lambda$-frame, \item $\frak{s}$ satisfies the conjugation property, \item the class of uniqueness triples satisfies the existence property, \item every saturated model in $\lambda^+$ is an amalgamation base and \item $(K,\preceq)$ satisfies the $(\lambda,\lambda^+)$-tameness for non-forking types over saturated models property.
\end{enumerate}
Then for every two models $M,M^+ \in K^{sat}$ the following holds: $$M \preceq M^+ \Leftrightarrow M \preceq^{NF}_{\lambda^+}M^+.$$
\end{theorem}

\begin{theorem}[Theorem \ref{corollary a good non-forking frame assuming tameness for non-forking types over saturated models}]\label{1}
Suppose:
\begin{enumerate}
\item $\frak{s}=(K,\preceq,S^{bs},\dnf)$ is a semi-good non-forking $\lambda$-frame,
\item $\frak{s}$ satisfies the conjugation property,  
\item every saturated model in $\lambda^+$ over $\lambda$ is an amalgamation base,
\item $(K,\preceq)$ satisfies the $(\lambda,\lambda^+)$-tameness for non-forking types over saturated models property and 
\item the class of uniqueness triples satisfies the existence property.
\end{enumerate}
Then $\frak{s^+}$ is a good non-forking $\lambda^+$-frame.
\end{theorem}

Let us describe the structure of the paper. In Section 2, we present sufficient conditions for $\frak{s}_{\lambda^+}$ and for $\frak{s}^{sat}$ being good non-forking $\lambda^+$-frames. In Sections 3 and 4, we show the contribution of tameness. In Section 5, we show the connection between the $(\lambda,\lambda^+)$-continuity of serial independence property and the symmetry axiom. In Section 6, we study the relation $\preceq^{NF}_{\lambda^+}$ as a replacement of $\preceq \restriction K_{\lambda^+}$. In Section 7, we present sufficient conditions for the equivalence between the relations $\preceq^{NF}_{\lambda^+}$ and $\preceq$. In Section 8, we apply the results of Section 7, proving the $(\lambda,\lambda^+)$-continuity of serial independence property. In Section 9, we apply results from previous sections, presenting sufficient conditions for $\frak{s}_{\lambda^+}$ and of $\frak{s^+}$ being good non-forking $\lambda^+$-frames. In Section 10, we show that even if the relations $\preceq^{NF}_{\lambda^+}$ and $\preceq$ are not equivalent, the definitions of `type' in the different AECs considered here, coincide.  
 
\section{Non-forking Frames}

Shelah \cite{shh}.III introduced the notion of a good non-forking $\lambda$-frame. It is an axiomatization of the non-forking relation in superstable first-order theories. In \cite[Definition 2.1.3]{jrsh875}, good non-forking frames generalized to semi-good non-forking frames: the stability hypothesis is weakened. 

Recall,
\begin{definition}\label{definition of a semi-good non-forking frame}
\emph{A (semi-)good non-forking $\lambda$-frame} is a quadruple, $(K,\preceq,S^{bs},\dnf)$ such that the following hold:
\begin{enumerate}
\item $(K,\preceq)$ is an AEC with $LST$-number $\lambda$ at most, satisfying the joint embedding in $\lambda$ and amalgamation in $\lambda$ properties, and $(K,\preceq)$ has no maximal model of cardinality $\lambda$, \item $S^{bs}$ is a function of $K_\lambda$ such that $S^{bs}(M)$ is a set of non-algebraic types; $S^{bs}$ is closed under isomorphisms; it satisfies density and basic (almost) stability [$|S^{bs}(M)|\leq \lambda^+$ for each model $M$ of cardinality $\lambda$] and \item $\dnf$ is closed under isomorphisms and satisfies the monotonicity, local character, uniqueness, symmetry, extension and continuity axioms.
\end{enumerate}
\end{definition}

\begin{definition}
An AEC in $\lambda$ is defined similarly to AEC, but its models are of cardinality $\lambda$. So it is closed under unions of increasing continuous sequences of length less then $\lambda^{+}$ only and the existence of a $LST$-number is irrelevant.
\end{definition}

\begin{definition} 
A (semi-)good non-forking $\lambda$-frame \emph{of the second version} is defined similarly to a (semi-)good non-forking $\lambda$-frame, except the following difference: $(K,\preceq)$ is an AEC in $\lambda$ (in place of an AEC).
\end{definition}

\begin{proposition}
Let $(K,\preceq)$ be an AEC. The following are equivalent:
\begin{enumerate}
\item there are $S^{bs},\dnf$ such that $(K,\preceq,S^{bs},\dnf)$ is a good non-forking $\lambda$-frame.
\item there are $S^{bs},\dnf$ such that $(K_\lambda,\preceq \restriction K_\lambda,S^{bs},\dnf)$ is a good non-forking $\lambda$-frame of the second version.
\end{enumerate}
\end{proposition}

\begin{proof}
By Fact \cite[1.0.18]{jrsh875}.
\end{proof}

From now on, we assume:
\begin{hypothesis}\label{hypothesis 1}
\mbox{}
\begin{enumerate} 
\item $(K,\preceq)$ is an AEC and
\item $\frak{s}=(K,\preceq,S^{bs},\dnf)$ is a semi-good non-forking $\lambda$-frame.
\end{enumerate}
\end{hypothesis}
 
\begin{remark}\label{local character almost implies continuity}
By \cite{jrsh940},
 without loss of generality, for each $M \in K_\lambda$ $S^{bs}(M)=S^{na}(M)$, namely, the basic types are the non-algebraic types. Anyway, we do not use it.
%In this case, the
%continuity axiom is an easy consequence of the local character.
\end{remark}

\begin{comment}
%%%%%%%%%%
%This is a comment only because I don't want it to delay me to
%complete my Ph.D.
In our work we study 1-types only. But what is
the relation between almost-stability in the sense of 1-types and
almost stability in the sense of n-types? For example, we do not
answer the following question:
\begin{question}
Is there an AEC $(K,\preceq)$ such that:
\begin{enumerate}
\item $LST(K,\preceq) \leq \lambda$. \item $(K_\lambda,\preceq
\reduction K_\lambda)$ satisfies the amalgamation property. \item
$(K,\preceq)$ is almost stable  in lambda. \item is not
n-type-stable for some finite n?
\end{enumerate}
\end{question}
Of course in such an AEC there is no saturated model. %%%%%%%%%%
%%%%%%%%%
\end{comment}

We recall \cite[Definition 2.6.1]{jrsh875}, where we extend the non-forking relation to include models of
cardinality greater than $\lambda$.

\begin{definition}\label{preparation for forking for big models}
\label{2.9} $\dnf^{\geq \lambda}$ is the class of quadruples
$(M_0,a,M_1,M_2)$ such that:
\begin{enumerate}
\item $\lambda \leq ||M_i||$ for each $i<3$.\item $M_0 \preceq M_1
\preceq M_2$ and $a \in M_2-M_1$. \item For some model $N_0 \in
K_\lambda$ with $N_0 \preceq M_0$ for each model $N \in
K_\lambda$, $N_0 \preceq N \preceq M_1 \Rightarrow
\dnf(N_0,a,N,M_2)$.
\end{enumerate}
\end{definition}

\begin{definition}\label{forking for big models}
Let $M_0,M_1$ be models in $K_{\geq \lambda}$ with $M_0 \preceq
M_1$ and $p \in S(M_1)$. We say that $p$ does not fork over $M_0$,
when for some triple $(M_1,M_2,a) \in p$ we have $\dnf^{\leq
\lambda}(M_0,a,M_1,M_2)$.
\end{definition}

\begin{remark}
We can replace the quantification `for some' ($M_1,M_2,a$) in
Definition \ref{forking for big models} by `for each'.
\end{remark}

\begin{definition}\label{basic for big models}\label{definition of basic types over models of greater cardinality}
Let $M \in K_{>\lambda},\ p \in S(M)$. $p$ is said to be
\emph{basic} when there is $N \in K_\lambda$ such that $N \preceq
M$ and $p$ does not fork over $N$. For every $M \in K_{>\lambda},\
S^{bs}_{> \lambda}(M)$ is the set of basic types over $M$.
Sometimes we write $S^{bs}_{\geq \lambda}(M)$, meaning $S^{bs}(M)$
or $S^{bs}_{> \lambda}(M)$ (the unique difference is the
cardinality of $M$). Similarly, we define $S^{bs}_{\lambda^+}$ and $\dnf_{\lambda^+}$.
\end{definition}

We now define the first candidate for a good non-forking $\lambda^+$-frame (of the second version). 
\begin{definition}
Let $\frak{s}$ be a semi-good non-forking $\lambda$-frame. We define $\frak{s}_{\lambda^+}=(K_{\lambda^+},\preceq \restriction K_{\lambda^+},S^{bs}_{\lambda^+}, \dnf_{\lambda^+})$.
\end{definition}

We restate \cite[Proposition 3.1.9(1)]{jrsh875} as:
\begin{fact}\label{fact no maximal in s lambda +}
Let $\frak{s}$ be a semi-good non-forking $\lambda$-frame. Then there is no maximal model of cardinality $\lambda^+$.
\end{fact}

The following fact  is an immediate consequence of \cite[Theorem 2.6.8]{jrsh875}.

\begin{fact}\label{theorem 2.6.8 of 875}
Let $\frak{s}$ be a semi-good non-forking $\lambda$-frame. Then $\frak{s}_{\lambda^+}$ satisfies the density, monotonicity, local character and continuity axioms.
\end{fact}

\begin{proposition}\label{fact the remain axioms}\label{the remain axioms}
If $\frak{s}_{\lambda^+}$ satisfies Axioms \ref{4 axioms} then it is a good non-forking $\lambda^+$-frame minus the joint embedding and amalgamation properties. 
\end{proposition}

\begin{proof}
By Fact \ref{theorem 2.6.8 of 875} and Fact \ref{fact no maximal in s lambda +}.
\end{proof}

 \begin{comment}
We assume freely (but later we will get it):
\begin{hypothesis}\label{hypothesis for sat}
\mbox\\
\begin{enumerate}
\item if $\delta<\lambda^{++}$ is a limit ordinal and $\langle M_\alpha:\alpha<\delta \rangle$ is a $\preceq$-increasing continuous sequence of models in $K^{sat}$ then $\bigcup_{\alpha<\delta} M_\alpha \in K^{sat}$ (or equivalently $(K^{sat},\preceq \restriction K^{sat})$ is an AEC in $\lambda^+$), and \item if $M \in K^{sat}$ and $p,q \in S(M)$ then $p=q$ in $(K,\preceq)$ if and only if $p=q$ as types in $(K^{sat},\preceq \restriction K^{sat})$( in Theorem \ref{the notions of type coincide} we will justify this hypothesis).
\end{enumerate}
\end{hypothesis}
\end{comment}
 
We now define the second candidate for a good non-forking $\lambda^+$-frame (of the second version): the `restriction' of $\frak{s}_{\lambda^+}$ to the saturated models in $\lambda^+$ over $\lambda$.
\begin{definition}
Let $\frak{s}$ be a semi-good non-forking $\lambda$-frame. We define $\frak{s}^{sat}=(K^{sat},\preceq \restriction K^{sat},S^{bs}_{\lambda^+} \restriction K^{sat}, \dnf_{\lambda^+} \restriction K^{sat})$, where $\dnf_{\lambda^+} \restriction K^{sat}$ is the class of quadruples in $\dnf_{\lambda^+}$ such that the three models are in $K^{sat}$.
\end{definition}

\begin{remark}\label{K sat is AEC iff}
The pair $(K^{sat},\preceq \restriction K^{sat})$ is an AEC in $\lambda^+$ if and only if for every limit ordinal $\delta<\lambda^{++}$ and every increasing continuous sequence, $\langle M_\alpha:\alpha<\delta \rangle$, of models in $K^{sat}$, we have $\bigcup_{\alpha<\delta} M_\alpha \in K^{sat}$.
\end{remark}

\begin{proposition}
Let $\delta$ be a limit ordinal with $\lambda<cf(\delta)$ and let $\langle N_\alpha:\alpha<\delta \rangle$ be an increasing continuous sequence of models in $K^{sat}$. Then $\bigcup_{\alpha<\delta}N_\alpha \in K^{sat}$.
\end{proposition}

\begin{proof}
Let $M$ be a model of cardinality $\lambda$ such that $M \preceq \bigcup_{\alpha<\delta}N_\alpha$. Since $\lambda<cf(\delta)$, for some $\beta<\delta$, $M \subseteq N_\beta$. Since $M \preceq \bigcup_{\alpha<\delta}N_\alpha$ and $N_\beta \preceq \bigcup_{\alpha<\delta}N_\alpha$, we have $M \preceq N_{\beta}$. But $N_\beta$ is saturated in $\lambda^+$ over $\lambda$. So every type over $M$ is realized in $N_\beta$.  
\end{proof}

In Fact \ref{fact existence of sat}, Proposition \ref{extending from lambda to sat} and Proposition \ref{extending to sat}, we do not have to assume Hypothesis \ref{hypothesis 1}, but only the following hypothesis:
\begin{hypothesis}\label{hypothesis basic types}
\mbox{}
\begin{enumerate}
\item $(K,\preceq)$ is an AEC and \item $K,\preceq$ and $S^{bs}$ satisfy Items (1) and (2) of Definition \ref{definition of a semi-good non-forking frame} (in particular, $|S^{bs}(M)| \leq \lambda^+$ holds for every $M$ of cardinality $\lambda$). 
\end{enumerate}
\end{hypothesis}

The following fact is a restatement of \cite[Theorem 2.5.8(2)]{jrsh875}. It presents a way to construct a saturated model in $\lambda^+$ over $\lambda$. Since we use only basic types, this fact is not trivial.
\begin{fact}\label{fact existence of sat} 
Suppose:
\begin{enumerate}
\item Hypothesis \ref{hypothesis basic types} holds,
\item $\langle M_\alpha:\alpha<\lambda^+ \rangle$ is an increasing continuous sequence of models of cardinality $\lambda$ and \item for every $\alpha<\lambda^+$ and every basic type $p$ over $M_\alpha$, $p$ is realized in some $M_\beta$ with $\alpha<\beta<\lambda^+$.
\end{enumerate}
Then $\bigcup_{\alpha<\lambda^+}M_\alpha$ is a saturated model in $\lambda^+$ over $\lambda$. 
\end{fact}

\begin{proposition}\label{extending from lambda to sat}
Assume Hypothesis \ref{hypothesis basic types}. 
Every model of cardinality $\lambda$ can be extended to a saturated model in $\lambda^+$ over $\lambda$.
\end{proposition}

\begin{proof}
Let $M_0$ be a model of cardinality $\lambda$. We choose $M_\alpha$ by induction on $\alpha \in (0,\lambda^+)$, such that the sequence $\langle M_\alpha:\alpha<\lambda^+ \rangle$ satisfies Conditions (1) and (2) of Fact \ref{fact existence of sat} (as in the proof of $(2) \rightarrow (3)$ in \cite[Theorem 2.5.8(2)]{jrsh875}). By Fact \ref{fact existence of sat}, $\bigcup_{\alpha<\lambda^+}M_\alpha$ is saturated in $\lambda^+$ over $\lambda$.
\end{proof}

\begin{proposition}\label{extending to K sat}\label{extending to sat}
Assume Hypothesis \ref{hypothesis basic types}.
Every model of cardinality $\lambda^+$ can be extended to a saturated model in $\lambda^+$ over $\lambda$.
\end{proposition}

\begin{proof}
Let $M_1 \in K_{\lambda^+}$. Take a model $M^- \in K$ of cardinality $\lambda$, such that $M^- \preceq M_1$. By Proposition \ref{extending from lambda to sat}, we can find a saturated model in $\lambda^+$ over $\lambda$, $M_2$ with $M^- \preceq M_2$. By \cite[Lemma II.1.14 (saturativity=model homogeneity)]{shh}, there is an embedding $f$ of $M_1$ into $M_2$ fixing $M^-$ pointwise. So $f[M_1] \preceq M_2$. Since every AEC is closed under isomorphisms, we conclude that $M_1$ can be extended to a saturated model in $\lambda^+$ over $\lambda$.
\end{proof}

\begin{remark}\label{remark meanings of types in section 2}
By Proposition \ref{a completion for section 2}, for every two types $p,q$ over a saturated model in $\lambda^+$ over $\lambda$, $p=q$ in the context of $(K,\preceq)$ if and only $p=q$ in the context of $(K^{sat},\preceq)$. We use it freely. 
\end{remark}

Proposition \ref{the remain axioms for sat} 
is the analog of Proposition \ref{the remain axioms} for $\frak{s}^{sat}$. A comparison shows that while in $\frak{s}_{\lambda^+}$, we have always an AEC in $\lambda^+$, in $\frak{s}^{sat}$ we have always the joint embedding property. 
\begin{proposition}\label{the remain axioms for sat}
Assume that $(K^{sat},\preceq \restriction K^{sat})$ is an AEC in $\lambda^+$. If $\frak{s}^{sat}$ satisfies Axioms \ref{4 axioms} then it is a good non-forking $\lambda^+$-frame minus amalgamation in $\lambda^+$.
\end{proposition}

\begin{proof}
By Fact \ref{fact no maximal in s lambda +} and Proposition \ref{extending to K sat}, $K^{sat}$ has no maximal model. Since $K^{sat}$ satisfies categoricity, it satisfies the joint embedding property. By Fact \ref{theorem 2.6.8 of 875} (and Remark \ref{remark meanings of types in section 2}), $\frak{s}^{sat}$ satisfies the density, monotonicity, local character and continuity axioms.
\end{proof}

\begin{proposition}\label{if s lambda + is a good nf frame and K sat is AEC then s sat is a good nf frame}
Assume that $(K^{sat},\preceq \restriction K^{sat})$ is an AEC in $\lambda^+$. If $\frak{s}_{\lambda^+}$ is a good non-forking $\lambda^+$-frame then $\frak{s}^{sat}$ is a good non-forking $\lambda^+$-frame. 
\end{proposition}

\begin{proof}
Since $\frak{s}_{\lambda^+}$ is a good non-forking $\lambda^+$-frame, the amalgamation in $\lambda^+$ property holds. In particular, every saturated model in $\lambda^+$ over $\lambda$ is an amalgamation base. By Proposition \ref{extending to K sat}, it is an amgalgamation base in the sense of $K^{sat}$.

Since $\frak{s}_{\lambda^+}$ is a good non-forking $\lambda^+$-frame, it satisfies Axioms \ref{4 axioms}. Hence,
by Remark \ref{remark meanings of types in section 2}, it is easy to prove that $\frak{s}^{sat}$ satisfies the extension, uniqueness and basic stability axioms.

$\frak{s}_{\lambda^+}$ satisfies the symmetry axiom and we have to prove that $\frak{s}^{sat}$ satisfies the symmetry axiom. Let $M_0,M_1,M_2$ be saturated models in $\lambda^+$ over $\lambda$, such that $M_0 \preceq M_1 \preceq M_2$. Let $a \in M_1-M_0$ such that $tp(a,M_0,M_1)$ is basic. Let $b \in M_2-M_1$ such that $tp(b,M_1,M_2)$ does not fork over $M_0$. Since $\frak{s}_{\lambda^+}$ satisfies the symmetry axiom,  we can find $M_2'$ and $M_3$ of cardinality $\lambda^+$ such that the following hold:
\begin{enumerate}
\item $M_0 \preceq M_2' \preceq M_3$, \item $M_2 \preceq M_3$, \item $b \in M_2'$ and \item $tp(a,M_2',M_3)$ does not fork over $M_0$.  
\end{enumerate}

\begin{displaymath}
\xymatrix{M_4 \ar[rr]^{f} && M_5 \ar[r]^{id} & M_6 \\
b \in M_2' \ar[rr]^{id} \ar[u]^{id} && M_3 \ar[u]^{id} \ni a \\
M_0 \ar[r]^{id} \ar[u]^{id} & a \in M_1 \ar[r]^{id} & M_2 \ar[u]^{id} \ni b
}
\end{displaymath}

It is not necessarily that the models $M_2'$ and $M_3$ are saturated in $\lambda^+$ over $\lambda$. But
by Proposition \ref{extending to sat}, there is a model $M_4 \in K^{sat}$ such that $M_2' \preceq M_4$. Since $\frak{s}_{\lambda^+}$ satisfies the extension axiom, we can find an amalgamation $(id_{M_3},f,M_5)$ of $M_3$ and $M_4$ over $M_2'$ (in particular, $f(b)=b$), such that $tp(a,f[M_4],M_5)$ does not fork over $M_2'$. Since $\frak{s}_{\lambda^+}$ satisfies \cite[Proposition 2.5.6 (the transitivity proposition)]{jrsh875}, $tp(a,f[M_4],M_5)$ does not fork over $M_0$. By Proposition \ref{extending to sat}, we can find a model $M_6 \in K^{sat}$ such that $M_5 \preceq M_6$.

We have the following:
\begin{enumerate} 
\item $f[M_4]$ and $M_6$ are saturated in $\lambda^+$ over $\lambda$,
\item $M_0 \preceq f[M_4] \preceq M_6$, \item $M_2 \preceq M_6$, \item $b \in f[M_4]$ and \item $tp(a,f[M_4],M_6)$ does not fork over $M_0$.  
\end{enumerate} 
The proof of the symmetry axiom is completed. Now by Proposition \ref{the remain axioms for sat}, $\frak{s}^{sat}$ is a good non-forking $\lambda^+$-frame.
\end{proof}

If the amalgamation property in $\lambda^+$ does not hold then $\frak{s}_{\lambda^+}$ is not a good non-forking $\lambda^+$-frame. In \cite{jrsh875}, we replace the relation $\preceq \restriction K_{\lambda^+}$ by $\preceq^{NF}_{\lambda^+}$, so the amalgamation property in $\lambda^+$ holds (this is Shelah's approach). The good non-forking $\lambda^+$-frame that is constructed in \cite{jrsh875} is called $\frak{s}^+$. The theorems in \cite{shh}.III relate to $\frak{s}^+$ and most of them are unknown for $\frak{s}_{\lambda^+}$. For example, the existence of primeness triples is proved in $\frak{s}^+$ (see \cite{jrprime}), but unknown for $\frak{s}_{\lambda^+}$, even if it is a good non-forking $\lambda^+$-frame.  

But in our main results, we present cases, where the relations $\preceq \restriction K_{\lambda^+}$ and $\preceq^{NF}_{\lambda^+}$ coincide. So we have the advantages of $\preceq$ and of $\preceq^{NF}_{\lambda^+}$. In particular, our results enable to apply the results and methods of \cite{shh}.III.

%%%%%%%%%%%%%%%%%%%%%%%%%%%%%%%%%%
%%               section   %%%%%%%
%%%%%%%%%%%%%%%%%%%%%%%%%%%%%%%%%%
\section{Variants of Tameness}

Given that it is not clear how to define tameness in three aspects of our context, we could offer
$2^3=8$ variants of tameness. Actually, we study five variants. In order to clarify the first aspect, we recall known facts about the connection between the amalgamation property and the definition of galois-types (due to Shelah). Note that since all the types in this paper are galois-types, we call them types, omitting `galois'. The following definition of types sums up \cite[Definitions 1.0.22 and 1.0.24]{jrsh875}:
\begin{definition}\label{definition of E^*}
Let $M_0,M_1,M_2$ be models in $K$ and let $a_1 \in M_1-M_0$ and $a_2 \in M_2-M_0$. We say that $(M_0,M_1,a_1)E^*(M_0,M_2,a_2)$ when there is an amalgamation $(f_1,f_2,M_3)$ of $M_1$ and $M_2$ over $M_0$ such that $f_1(a_1)=f_2(a_2)$. We define $E$ as the transitive closure of $E^*$. $tp(a_1,M_0,M_1)=tp(a_2,M_0,M_2)$ if and only if $(M_0,M_1,a_1)E(M_0,M_2,a_2)$. 
\end{definition}
$(M_0,M_1,a_1)E^*(M_0,M_2,a_2)$ says that `$a_1$ and $a_2$ strongly realize the same type over $M_0$ (in the definitions of \cite{babook}). When the amalgamation property in $\lambda$ holds, the restriction of $E^*$ to models of cardinality $\lambda$ is a transitive partial order, so if $||M_0||=\lambda$ then $tp(a_1,M_0,M_1)=tp(a_2,M_0,M_2)$ if and only if the triples $(M_0,M_1,a_1)$ and $(M_0,M_2,a_2)$ are $E^*$-equivalent. 

\begin{comment}
%I delete it only because the referee ask me to delete definitions (so I delete the second definition, and now I can't compare to it).
In the definition of strong tameness (Definition \ref{definition of strong tameness}), we conclude that the triples are $E^*$-equivalent (equivalently, the elements strongly realize the same type). But in the definition of tameness (Definition \ref{definition of tameness}), we conclude that the elements only realize the same type. 
\end{comment}

%Generally, when the amalgamation property is not assumed, we should distinguish between realizing a type and strongly %realizing a type (it is defined in \cite{babook}, but in the definitions below, in place of saying `strongly realizing a type' , we %use the relation $E^*$ \cite[Definition 1.0.22]{jrsh875}). Here, in the definition of tameness, we conclude that the elements %strongly realize the same type. But in the definition of weak tameness, we conclude that the elements only realize the same %type. 
\begin{definition}\label{definition of strong tameness}
$(K,\preceq)$ is said to satisfy \emph{the strong $(\lambda,\lambda^+)$-tameness property} when for every $M_0,M_1,M_2 \in K_{\lambda^+}$ with $M_0 \preceq M_1$ and $M_0 \preceq M_2$, the following condition holds: For every $a_1 \in M_1-M_0$ and $a_2 \in M_2-M_0$, if for every $M_0^- \in K_\lambda$ with $M_0^- \preceq M_0$, we have $tp(a_1,M_0^-,M_1)=tp(a_2,M_0^-,M_2)$ then $(M_0,M_1,a_1)E^*(M_0,M_2,a_2)$.
\end{definition} 

\begin{comment}
\begin{definition}\label{definition of tameness}
$(K,\preceq)$ is said to satisfy \emph{the $(\lambda,\lambda^+)$-tameness property} when for every $M_0,M_1,M_2 \in K_{\lambda^+}$ with $M_0 \preceq M_1$ and $M_0 \preceq M_2$, the following condition holds: For every $a_1 \in M_1-M_0$ and $a_2 \in M_2-M_0$, if for every $M_0^- \in K_\lambda$ with $M_0^- \preceq M_0$, we have $tp(a_1,M_0^-,M_1)=tp(a_2,M_0^-,M_2)$ then $tp(a_1,M_0,M_1)=tp(a_2,M_0,M_2)$.
\end{definition} 
\end{comment}

\begin{remark}\label{tameness + amalgamation implies strong tameness}
Obviously, if the $(\lambda,\lambda^+)$-tameness property and the amalgamation property in $\lambda^+$ hold then the strong $(\lambda,\lambda^+)$-tameness property holds. 
\end{remark}

In the following definition, the quantifier of $M_0^-$ is `for some' in place of `for every'. But see Proposition \ref{strong tameness implies for non-forking}.  
\begin{definition}\label{definition of strong tameness for non-forking types}
$(K,\preceq)$ is said to satisfy \emph{the strong $(\lambda,\lambda^+)$-tameness for non-forking types property} when for every $M_0,M_1,M_2 \in K_{\lambda^+}$ with $M_0 \preceq M_1$ and $M_0 \preceq M_2$, the following condition holds: For every $a_1 \in M_1-M_0$ and $a_2 \in M_2-M_0$, if for some $M_0^- \in K_\lambda$ with $M_0^- \preceq M_0$, we have $tp(a_1,M_0^-,M_1)=tp(a_2,M_0^-,M_2)$ and the types $tp(a_1,M_0,M_1)$ and $tp(a_2,M_0,M_2)$ do not fork over $M_0^-$ then $(M_0,M_1,a_1)E^*(M_0,M_2,a_2)$.
\end{definition} 

[Definition \ref{definition of strong tameness for non-forking types} is similar to the uniqueness property in a semi-good non-forking frame. But the models are not of the same cardinality: while $||M_0^-||=\lambda$, $||M_0||=\lambda^+$].

\begin{proposition}\label{strong tameness implies for non-forking}\label{strong tameness implies strong tameness (for non-forking types)}
If $(K,\preceq)$ satisfies the strong $(\lambda,\lambda^+)$-tameness property, then it satisfies the strong $(\lambda,\lambda^+)$-tameness for non-forking types property.
\end{proposition}

The proof of Proposition \ref{strong tameness implies for non-forking} is rather easy, but for completeness we give it.
\begin{proof}
Let $M_0,M_1,M_2 \in K_{\lambda^+}$ such that $M_0 \preceq M_1$ and $M_0 \preceq M_2$. Let $a_1 \in M_1-M_0$ and $a_2 \in M_2-M_0$ be given. Assume that for some $M_0^- \in K_\lambda$ with $M_0^- \preceq M_0$, we have $tp(a_1,M_0^-,M_1)=tp(a_2,M_0^-,M_2)$ and the types $tp(a_1,M_0,M_1)$ and $tp(a_2,M_0,M_2)$ do not fork over $M_0^-$. We should prove that $(M_0,M_1,a_1)E^*(M_0,M_2,a_2)$.

Since $(K,\preceq)$ satisfies the strong $(\lambda,\lambda^+)$-tameness property, it is sufficient to prove that for every $M^* \in K_\lambda$ with $M^* \preceq M_0$, we have $tp(a_1,M^*,M_1)=tp(a_2,M^*,M_2)$ (by Definition \ref{definition of strong tameness}, where $M^*$ stands for the $M_0^-$). 

Since $LST(K,\preceq) \leq \lambda$, we can find $N \in K_\lambda$ satisfying $M_0^- \cup M^* \subseteq N \preceq M_0$. Clearly, $M_0^- \preceq N$ and $M^* \preceq N$. Since the types $tp(a_1,M_0,M_1)$ and $tp(a_2,M_0,M_2)$ do not fork over $M_0^-$, the types $tp(a_1,N,M_1)$ and $tp(a_2,N,M_2)$ do not fork over $M_0^-$ . So by uniqueness of non-forking, 
$tp(a_1,N,M_1)=tp(a_2,N,M_2)$. So $tp(a_1,M^*,M_1)=tp(a_2,M^*,M_2)$.
\end{proof}

\begin{definition}\label{definition of tameness for non-forking types}
\emph{The $(\lambda,\lambda^+)$-tameness for non-forking types property} is similar to the strong $(\lambda,\lambda^+)$-tameness for non-forking types property, but we conclude only $tp(a_1,M_0,M_1)=tp(a_2,M_0,M_2)$.
\end{definition}

\begin{remark}\label{tameness for non-forking types + amalgamation implies strong tameness for non-forking types}
Obviously, if the $(\lambda,\lambda^+)$-tameness for non-forking types property and the amalgamation property in $\lambda^+$ hold then the strong $(\lambda,\lambda^+)$-tameness for non-forking types property holds. 
\end{remark}

Proposition \ref{strong tameness implies amalgamation (for non-forking types)} is the converse of Remark \ref{tameness for non-forking types + amalgamation implies strong tameness for non-forking types}. Claim \ref{strong tameness implies E^*} is a preparation for Proposition \ref{strong tameness implies amalgamation (for non-forking types)}.

\begin{claim}\label{strong tameness implies E^*}
Suppose that the strong $(\lambda,\lambda^+)$-tameness for non forking types property holds. Then for every two models $N_0,N_1 \in K_{\lambda^+}$ with $N_0 \preceq N_1$ and every two elements $a,b \in N_1-N_0$, if $tp(a,N_0,N_1)=tp(b,N_0,N_1)$ and it is basic then $(N_0,N_1,a)E^*(N_0,N_1,b)$. 
\end{claim}

\begin{proof}
Take a model $M$ of cardinality $\lambda$ such that the type $tp(a,N_0,N_1)$ does not fork over $M$. Clearly, $tp(a,M,N_1)=tp(b,M,N_1)$. So by strong $(\lambda,\lambda^+)$-tameness for non-forking types, $(N_0,N_1,a)E^*(N_0,N_1,b)$.
\end{proof}

The following proposition is due to Boney.\footnote{in a private communication, during the ICM (International Congress of Mathematicians) 2014 satellite meeting on Classification theory and its applications on Seoul, Korea}

\begin{proposition}\label{strong tameness implies amalgamation (for non-forking types)}
Assume that the amalgamation property in $\lambda$ holds. Then the strong $(\lambda,\lambda^+)$-tameness for non-forking types property yields the amalgamation property in $\lambda^+$.
\end{proposition}

\begin{proof}
Let $N_0,N_1,N_2$ be models in $K$ of cardinality $\lambda^+$ such that $N_0 \prec N_1$ and $N_0 \prec N_1$. We should find an amalgamation of $N_1$ and $N_2$ over $N_0$. Take an element $a \in N_1-N_0$ such that $p:=tp(a,N_0,N_1)$ is basic. Take a non-forking extension $q \in S^{bs}(N_2)$ of $p$. Take a model $N_3$ and an element $b \in N_3$ such that $N_2 \preceq N_3$ and $tp(b,N_2,N_3)=q$. But $q \restriction N_0=p$. So by Claim \ref{strong tameness implies E^*}, $(N_0,N_3,b)E^*(N_0,N_1,a)$. Hence, there is an amalgamation $(f,g,N_4)$ of $N_3$ and $N_1$ over $N_0$. So $(f \restriction N_2,g,N_4)$ is an amalgamation of $N_2$ and $N_1$ over $N_0$. 
\end{proof}

\begin{proposition}\label{strong iff amalgamation (for non-forking types)}
Assume that the amalgamation property in $\lambda$ holds.
Strong $(\lambda,\lambda^+)$-tameness for non-forking types is equivalent to the conjunction of amalgamation in $\lambda^+$ and $(\lambda,\lambda^+)$-tameness for non-forking types.
\end{proposition}

\begin{proof}
By Remark \ref{tameness for non-forking types + amalgamation implies strong tameness for non-forking types} and Proposition \ref{strong tameness implies amalgamation (for non-forking types)}.
\end{proof}

In all the variants of tameness appearing below, adding the word `strong' is equivalent to assuming that every `relevant' model in $\lambda^+$ is an amalgamation base.

The following proposition is an analog of Proposition \ref{strong tameness implies amalgamation (for non-forking types)}.
\begin{proposition}\label{strong tameness implies amalgamation}
Assume that the amalgamation property in $\lambda$ holds. Then the strong $(\lambda,\lambda^+)$-tameness property yields the amalgamation property in $\lambda^+$.
\end{proposition}

\begin{proof}
By Propositions \ref{strong tameness implies strong tameness (for non-forking types)} and \ref{strong tameness implies amalgamation (for non-forking types)}.
\end{proof}

The following proposition is an analog of Proposition \ref{strong iff amalgamation (for non-forking types)}.
\begin{proposition}\label{strong iff amalgamation}
Assume that the amalgamation property in $\lambda$ holds.
Strong $(\lambda,\lambda^+)$-tameness is equivalent to the conjunction of amalgamation in $\lambda^+$ and $(\lambda,\lambda^+)$-tameness.
\end{proposition}

\begin{proof}
Mainly, by Proposition \ref{strong tameness implies amalgamation}.
\end{proof}

The saturated models in $\lambda^+$ over $\lambda$ are natural candidates for base amalgamations: If we restrict ourselves to the saturated models in $\lambda^+$ over $\lambda$ then we have categoricity in $\lambda^+$ \cite[Theorem 1.0.32]{jrsh875}. By a well-known theorem of Shelah, if an AEC is categorical in $\lambda^+$ but does not satisfy the amalgamation property in $\lambda^+$ then under plausible set theoretic assumptions, we can prove the existence of $2^{\lambda^{++}}$ models of cardinality $\lambda^{++}$. So it is natural to assume that every saturated model in $\lambda^+$ over $\lambda$ is an amalgamation base. This leads to the next definition.

\begin{definition}\label{definition of tameness over saturated models}
$(K,\preceq)$ is said to satisfy \emph{the $(\lambda,\lambda^+)$-tameness for non-forking types over saturated models property} when for every $M_0,M_1,M_2$ of cardinality $\lambda^+$ such that $M_0$ is saturated in $\lambda^+$ over $\lambda$, $M_0 \preceq M_1$ and $M_0 \preceq M_2$ the condition in Definition \ref{definition of tameness for non-forking types} holds. 
\end{definition}

\begin{definition}
The strong $(\lambda,\lambda^+)$-tameness for non-forking types over saturated models is defined as expected.
\end{definition} 

The proof of Proposition \ref{... implies strong for saturated} is easy. We apply Proposition \ref{... implies strong for saturated} in the proof of Theorem \ref{the main theorem of the paper}. 
\begin{proposition}\label{... implies strong for saturated}
 Suppose:
\begin{enumerate}
\item the amalgamation property in $\lambda$ holds,
\item Every saturated model in $\lambda^+$ over $\lambda$ is an amalgamation base and \item the $(\lambda,\lambda^+)$-tameness for non-forking types over saturated models property.
\end{enumerate}
Then the strong $(\lambda,\lambda^+)$-tameness for non-forking types over saturated models property holds.
\end{proposition}

Proposition \ref{strong iff amalgamation (for non-forking types over saturated models)} is an analog of Proposition \ref{strong iff amalgamation (for non-forking types)}.
\begin{proposition}\label{strong iff amalgamation (for non-forking types over saturated models)}
Assume that the amalgamation property in $\lambda$ holds.
Then the strong $(\lambda,\lambda^+)$-tameness for non-forking types over saturated models property is equivalent to the conjunction of the following two properties:
\begin{enumerate}
\item Every saturated model in $\lambda^+$ over $\lambda$ is an amalgamation base and \item the $(\lambda,\lambda^+)$-tameness for non-forking types over saturated models property.
\end{enumerate}
\end{proposition}

\begin{proof}
Similar to the proof of Proposition \ref{strong iff amalgamation (for non-forking types)}.
\end{proof}

\begin{remark}
In Proposition \ref{strong iff amalgamation (for non-forking types over saturated models)}, we may omit the words `non-forking' (appearing twice).
\end{remark}

\section{Deriving Non-Forking Frames Using Tameness}

\begin{theorem}\label{the main theorem}\label{symmetry implies good-frame}\label{symmetry implies good non-forking frame}
Suppose:
\begin{enumerate}
\item $\frak{s}$ is a semi-good non-forking $\lambda$-frame, \item the $(\lambda,\lambda^+)$-tameness for non-forking types holds and \item $\frak{s}_{\lambda^+}$ satisfies symmetry.
\end{enumerate} 
Then $\frak{s}_{\lambda^+}$ is a good non-forking $\lambda^+$-frame minus the joint embedding and amalgamation properties.
\end{theorem}

\begin{proof}
By Proposition \ref{the remain axioms} and Claims \ref{uniqueness in s lambda +},  \ref{extension in s lambda +} and \ref{stability in s lambda +}. Note that these claims are not new: similar claims appear in \cite{jrsh875}. But for completeness, we give their proofs. 
\end{proof}

Theorem \ref{symmetry implies good-frame for sat} is the analog of Theorem \ref{symmetry implies good-frame} for $\frak{s}^{sat}$.
\begin{theorem}\label{symmetry implies good-frame for sat}
Suppose:
\begin{enumerate}
\item $\frak{s}$ is a semi-good non-forking $\lambda$-frame, \item $(K^{sat},\preceq \restriction K^{sat})$ is an AEC, \item the $(\lambda,\lambda^+)$-tameness for non-forking types over saturated models holds and \item $\frak{s}^{sat}$ satisfies symmetry.
\end{enumerate} 
Then $\frak{s}^{sat}$ is a good non-forking $\lambda^+$-frame minus amalgamation in $\lambda^+$.
\end{theorem}

\begin{proof}
By Proposition \ref{the remain axioms for sat} and Claims \ref{uniqueness in s sat},  \ref{extension in s sat} and \ref{stability in s sat}. 
\end{proof}

\begin{claim}\label{uniqueness in s lambda +}
Assume that the $(\lambda,\lambda^+)$-tameness for non-forking types holds.
Then $\frak{s}_{\lambda^+}$ satisfies uniqueness.
\end{claim}

\begin{proof}
Let $N_0,N_1$ be two models in $K$ of cardinality $\lambda^+$ with $N_0 \preceq N_1$. Let $p,q$ be two types over $N_1$ which do not fork over $N_0$, such that $p \restriction N_0=q \restriction N_0$. We should prove that $p=q$. Let $M_p,M_q$ be models in $K$ of cardinality $\lambda$ such that $M_p \preceq N_0$, $M_q \preceq N_0$ and the types $p \restriction N_0,q \restriction N_0$ do not fork over $M_p,M_q$ respectively. Since $LST(K,\preceq) \leq \lambda$, we can find a model $M$ in $K$ of cardinality $\lambda$ such that $M_p \cup M_q \subseteq M \preceq N_0$. So $M_p \preceq N$ and $M_q \preceq N$. Therefore by monotonicity, the types $p \restriction N_0$ and $q \restriction N_0$ do not fork over $M$. Hence, by transitivity, the types $p$ and $q$ do not fork over $M$. Since $p \restriction M=q \restriction M$, by $(\lambda,\lambda^+)$-tameness for non-forking types, we have $p=q$.
\end{proof}

\begin{claim}\label{uniqueness in s sat}
Assume that the $(\lambda,\lambda^+)$-tameness for non-forking types over saturated models holds.
Then $\frak{s}^{sat}$ satisfies uniqueness.
\end{claim}

\begin{proof}
By the proof of Claim \ref{uniqueness in s lambda +}.
\end{proof}

The following claim is a preparation for extension. Note that it holds even without any remnant of tameness.
\begin{claim}\label{extension from lambda to lambda^+}
If $M$ is a model of cardinality $\lambda$, $N$ is a model of cardinality $\lambda^+$, $M \preceq N$ and $p \in S^{bs}(M)$ then there is a non-forking extension of $p$ to a type over $N$.
\end{claim}

\begin{proof}
%I have to replace each N by M^+ and each N_{\lambda^+} by N^+.
Take a filtration $\langle M_\alpha:\alpha<\lambda^+ \rangle$ of $N$ with $M_0=M$. Let $N_0$ be a model of cardinality $\lambda$ such that $M_0 \preceq N_0$ and for some $a \in N_0-M_0$ $tp(a,M_0,N_0)=p$. We choose by induction on $\alpha<\lambda^+$ a model $N_\alpha$ and an embedding $f_\alpha:M_\alpha \to N_\alpha$ such that: 
\begin{enumerate} 
\item $f_0$ is the identity from $M$ to $N_0$,
\item if $\alpha=\beta+1$ then $f_\beta \subseteq f_\alpha$ and $tp(a,f_\alpha[M_\alpha],N_\alpha)$ does not fork over $f_\beta[M_\beta]$ (it is possible by the extension property in $\frak{s}$) and \item if $\alpha$ is a limit ordinal then $N_\alpha=\bigcup_{\beta<\alpha}N_\alpha$ and $f_\alpha=\bigcup_{\beta<\alpha}f_\beta$.  
\end{enumerate}
We can prove by induction on $\alpha<\lambda^+$, that $tp(a,f_\alpha[M_\alpha],N_\alpha)$ does not fork over $M$: Assume that $tp(a,f_\beta[M_\beta],N_\beta)$ does not fork over $M$ for every $\beta<\alpha$. We have to prove that it holds for $\alpha$. If $\alpha=0$ then it holds by definition. If $\alpha=\beta+1$ for some $\beta$ then by Clause (2), $tp(a,f_\alpha[M_\alpha],N_\alpha)$ does not fork over $f_\beta[M_\beta]$. So by transitivity (in $\frak{s}$), $tp(a,f_\alpha[M_\alpha],N_\alpha)$ does not fork over $M$. If $\alpha$ is limit then it holds by continuity (in $\frak{s}$).  

Define $N_{\lambda^+}=:\bigcup_{\alpha<\lambda^+}N_\alpha$ and $f=:\bigcup_{\alpha<\lambda^+}f_\alpha$. 
Since $\dnf$ is closed under isomorphisms, it is sufficient to prove that $tp(a,f[N],N_{\lambda^+})$ does not fork over $M$. Let $M'$ be a model of cardinality $\lambda$ with $M \preceq M' \preceq f[N]$. For some $\alpha<\lambda^+$, we have $M' \subseteq f[M_\alpha]$. But $tp(a,f[M_\alpha],f[N])$ does not fork over $M$. So by monotonicity, $tp(a,M',f[N])$ does not fork over $M$.
\end{proof}

\begin{claim}\label{extension in s lambda +}
Assume that the $(\lambda,\lambda^+)$-tameness for non-forking types holds.
Then $\frak{s}_{\lambda^+}$ satisfies extension.
\end{claim}

\begin{proof}
Let $M,N$ be models in $K$ of cardinality $\lambda^+$ with $M \prec N$. Let $p$ be a basic type over $M$. We should find a type over $N$, extending $p$, that does not fork over $M$. Since $p$ is basic, there is a model $M_0$ of cardinality $\lambda$ such that $M_0 \preceq M$ and $p$ does not fork over $M_0$. By Claim \ref{extension from lambda to lambda^+}, there is a basic type $q$ over $N$, extending $p \restriction M_0$ such that $q$ does not fork over $M_0$. By monotonicity, $q$ does not fork over $M$ and $q \restriction M$ does not fork over $M_0$. So by tameness for non-forking types, $q \restriction M=p$. Hence, $q$ is a non-forking extension of $p$ to $N$.
\end{proof}

\begin{claim}\label{extension in s sat}
Assume that the $(\lambda,\lambda^+)$-tameness for non-forking types over saturated models holds.
Then $\frak{s}^{sat}$ satisfies extension.
\end{claim}

\begin{proof}
By the proof of Claim \ref{extension in s lambda +}.
\end{proof}

\begin{claim}\label{stability in s lambda +}
Assume that the $(\lambda,\lambda^+)$-tameness for non-forking types holds.
Then $\frak{s}_{\lambda^+}$ satisfies basic stability.
\end{claim}

There is no significant difference between the following proof and the proof of \cite[Proposition 10.1.10]{jrsh875}. 
\begin{proof}
By basic stability in $\frak{s}$ and $(\lambda,\lambda^+)$-tameness for non-forking types. We elaborate: Let $N$ be a model in $K$ of cardinality $\lambda^+$. We should prove that the number of basic types over $N$ is $\lambda^+$ at most. Let $\langle M_\alpha:\alpha<\lambda^+ \rangle$ be a filtration of $N$. For every basic type $p$ over $N$ we define $\alpha_p$ as the minimal ordinal $\alpha<\lambda^+$ such that $p$ does not fork over $M_\alpha$. We define $q_p:=p \restriction M_{\alpha_p}$. By basic almost stability (in $\frak{s}$), $||S^{bs}(M_\alpha)|| \leq \lambda^+$. So $|\{(\alpha_p,q_p):p$ is a basic type over $N\}| \leq \lambda^+ \times \lambda^+=\lambda^+$. So it is sufficient to prove that the function $p \to (\alpha_p,q_p)$ is an injection. Let $p_1,p_2$ be two basic types over $N$ such that $\alpha_{p_1}=\alpha_{p_2}$ and $q_{p_1}=q_{p_2}$. Denote $\alpha:=\alpha_{p_1}$. The types $p_1$ and $p_2$ do not fork over $\alpha$. But $p_1 \restriction M_\alpha=q_{p_1}=q_{p_2}=p_2 \restriction M_{\alpha}$. Hence, by $(\lambda,\lambda^+)$-tameness for non-forking types, $p_1=p_2$.
\end{proof}

\begin{claim}\label{stability in s sat}
Assume that the $(\lambda,\lambda^+)$-tameness for non-forking types over saturated models holds.
Then $\frak{s}^{sat}$ satisfies basic stability.
\end{claim}

\begin{proof}
By the proof of Claim \ref{stability in s lambda +}.
\end{proof}

\section{Continuity Yields Symmetry}

We generalize a bit the definition of independence 
\cite[Definition 3.2]{jrsi3}: here, we do not limit the cardinalities of the models and of the set $J$.

\begin{definition} \label{1.6} \label{definition of independence}
Let $\alpha^*$ be an ordinal.

(a) $\langle M_\alpha,a_\alpha:\alpha<\alpha^* \rangle ^\frown
\langle M_{\alpha^*} \rangle$ is said to be \emph{independent}
over $M$ when:
\begin{enumerate}
\item $\langle M_\alpha:\alpha \leq \alpha^* \rangle$ is an
increasing continuous sequence of models in $K$. \item $M
\preceq M_0$. \item For every $\alpha<\alpha^*$, $a_\alpha \in
M_{\alpha+1}-M_\alpha$ and the type
$tp(a_\alpha,M_\alpha,M_{\alpha+1})$ does not fork over $M$.
\end{enumerate}

(b) $\langle a_\alpha:\alpha<\alpha^* \rangle$ is said to be
\emph{independent} in $(M,M_0,N)$ when $M \preceq M_0 \preceq N$,
$\{a_\alpha:\alpha<\alpha^*\} \subseteq N-M$ and for some
increasing continuous sequence $\langle M_\alpha:0<\alpha \leq
\alpha^* \rangle$ and a model $N^+$ the sequence $\langle
M_\alpha,a_\alpha:\alpha<\alpha^* \rangle ^\frown \langle
M_{\alpha^*} \rangle$ is independent over $M$, $N \preceq N^+$ and
$M_{\alpha^*} \preceq N^+$.

(c) $\langle a_\alpha:\alpha<\alpha^* \rangle$ is said to be
 $K^{sat}$-\emph{independent} in $(M,M_0,N)$, when in addition, the models $M,N$ and $M_\alpha$ for each $\alpha<\alpha^*$ are in $K^{sat}$.

(d) When $M=M_0$, we may omit it.

\end{definition}

Using the independence terminology, we present two reformulations of symmetry:
\begin{remark}\label{remark reformulations of symmetry using independence}
The following are equivalent:
\begin{enumerate} 
\item $\frak{s}$ satisfies the symmetry axiom. 
\item For every $M,N \in K_\lambda$ with $M \preceq N$ and for every two elements $a,b \in N$ the sequence $\langle a,b \rangle$ is independent in $(M,N)$ if and only if the sequence $\langle b,a \rangle$ is independent in $(M,N)$. \item For every $M_0,M_1,M_2 \in K_\lambda$ with $M_0 \preceq M_1 \preceq M_2$ and for every two elements $a_0,a_1 \in M_2$ if $tp(a_0,M_0,M_1)$ is basic and $tp(a_1,M_1,M_2)$ does not fork over $M_0$ then the sequence $\langle a_1,a_0 \rangle$ is independent in $(M_0,M_2)$.
\end{enumerate}
\end{remark}

\begin{definition}\label{definition of the continuity of serial independence}
Let $\beta^*<\lambda^+$. 

(1) \emph{The $(\lambda,\lambda^+)$-continuity of independence of sequences of length $\beta^*$ property} is the following property:
Let $M \in K_{\lambda^+}$, $M \preceq N \in K_{\lambda^+}$ and let $\langle M_\alpha:\alpha<\lambda^+ \rangle$ be a filtration of $M$. If $\langle a_\beta:\beta<\beta^* \rangle$ is independent in $(M_\alpha,N)$ for each $\alpha<\lambda^+$ then it is independent in $(M,N)$.

(2) \emph{The $(\lambda,\lambda^+)$-continuity of serial independence property} means the $(\lambda,\lambda^+)$-continuity of independence of sequences of length $\beta^*$ property for every $\beta^*<\lambda^+$.

(3) if we say in (1) or (2) $K^{sat}$-independence then independent is replaced by $K^{sat}$-independent. 
\end{definition}

\begin{theorem}\label{continuity implies good frame}
Suppose:
\begin{enumerate}
\item $\frak{s}$ is a semi-good non-forking $\lambda$-frame, \item the $(\lambda,\lambda^+)$-tameness for non-forking types property holds, and \item the $(\lambda,\lambda^+)$-continuity of independence of sequences of length $2$ property holds.
\end{enumerate}
Then $\frak{s}_{\lambda^+}$  is a good non-forking $\lambda^+$-frame minus the joint embedding and amalgamation properties.
\end{theorem}

\begin{proof}
By Proposition \ref{continuity implies symmetry}, $\frak{s}_{\lambda^+}$ satisfies the symmetry axiom. So by Theorem \ref{symmetry implies good non-forking frame}, $\frak{s}_{\lambda^+}$ is a good non-forking $\lambda^+$-frame minus amalgamation in $\lambda^+$.
\end{proof}

Theorem \ref{continuity implies good frame for sat} is the analog of Theorem \ref{continuity implies good frame} for $\frak{s}^{sat}$.
\begin{theorem}\label{continuity implies good frame for sat}
Suppose:
\begin{enumerate} 
\item $\frak{s}$ is a semi-good non-forking $\lambda$-frame, \item $(K^{sat},\preceq \restriction K^{sat})$ is an AEC in $\lambda^+$, \item the $(\lambda,\lambda^+)$-tameness for non-forking types over saturated models in $\lambda^+$ over $\lambda$ property holds and \item the $(\lambda,\lambda^+)$-continuity of $K^{sat}$-independence of sequences of length $2$ property holds.
\end{enumerate}
Then $\frak{s}^{sat}$  is a good non-forking $\lambda^+$-frame minus amalgamation in $\lambda^+$.
\end{theorem}

\begin{proof}
By the proof of Proposition \ref{continuity implies symmetry}, the symmetry axiom holds. So by Theorem \ref{symmetry implies good-frame for sat}, $\frak{s}^{sat}$ is a good non-forking $\lambda^+$-frame minus amalgamation in $\lambda^+$.
\end{proof}
\begin{proposition}\label{continuity implies symmetry}
If the $(\lambda,\lambda^+)$-continuity of independence of sequences of length $2$ property holds then $\frak{s}_{\lambda^+}$ satisfies the symmetry axiom.
\end{proposition}

\begin{proof}
\begin{displaymath}
\xymatrix{a_1 \in M_{2,0} \ar[rr]^{id} && M_{2,\alpha} \ar[r]^{id} & M_{2,\alpha+1} \ar[rr]^{id} && N_2 \ni a_1 \\  
a_0 \in M_{1,0} \ar[rr]^{id} \ar[u]^{id} && M_{1,\alpha} \ar[r]^{id} \ar[u]^{id} & M_{1,\alpha+1} \ar[rr]^{id} \ar[u]^{id} && N_1 \ni a_0 \ar[u]^{id}\\
M_{0,0} \ar[rr]^{id} \ar[u]^{id} && M_{0,\alpha} \ar[r]^{id} \ar[u]^{id} & M_{0,\alpha+1} \ar[rr]^{id} \ar[u]^{id} && N_0 \ar[u]^{id}
}
\end{displaymath}

Let $N_0,N_1,N_2$ be three models in $K$ of cardinality $\lambda^+$. Let $a_0$ be an element in $N_1-N_0$ and let $a_1$ be an element in $N_2-N_1$. Suppose $tp(a_0,N_0,N_1)$ is basic and $tp(a_1,N_1,N_2)$ does not fork over $N_0$. By Remark \ref{remark reformulations of symmetry using independence}($3 \rightarrow 1$), it is sufficient to prove that the sequence $\langle a_1,a_0 \rangle$ is independent in $(N_0,N_2)$. Let $\langle M_{0,\alpha}:\alpha<\lambda^+ \rangle$, $\langle M_{1,\alpha}:\alpha<\lambda^+ \rangle$, $\langle M_{2,\alpha}:\alpha<\lambda^+ \rangle$ be filtrations of $N_0,N_1,N_2$ respectively, such that $a_0 \in M_{1,0}$ and $a_1 \in M_{2,0}$. For some club $E$ of $\lambda^+$, for every $\alpha \in E$, we have $M_{0,\alpha} \preceq M_{1,\alpha} \preceq M_{2,\alpha}$. By renaming, without loss of generality, it holds for $E=\lambda^+$.

For some $\alpha_0<\lambda^+$ the type $tp(a_0,N_0,N_1)$ does not fork over $M_{0,\alpha_0}$. Similarly, for some $\alpha_1<\lambda^+$ the type $tp(a_1,N_1,N_2)$ does not fork over $M_{0,\alpha_1}$. Define $\alpha^*=:max\{\alpha_1,\alpha_2\}$. By monotonicity, $tp(a_0,N_0,N_1)$ does not fork over $M_{0,\alpha^*}$ and $tp(a_1,N_1,N_2)$ does not fork over $M_{0,\alpha^*}$.  By renaming, without loss of generality $\alpha^*=0$. Let $\alpha<\lambda^+$. By monotonicity, $tp(a_0,M_{0,\alpha},M_{1,\alpha})$ is basic and $tp(a_1,M_{1,\alpha},M_{2,\alpha})$ does not fork over $M_{0,\alpha}$. So $\langle a_0,a_1 \rangle$ is independent in $(M_{0,\alpha},N_2)$. Since $\frak{s}$ satisfies the symmetry axiom, by Remark \ref{remark reformulations of symmetry using independence}($1 \rightarrow 2$), the sequence  $\langle a_1,a_0 \rangle$ is independent in $(M_{0,\alpha},N_2)$ (for every $\alpha<\lambda^+$). Hence, by $(\lambda,\lambda^+)$-continuity of sequences of length $2$, the sequence $\langle a_1,a_0 \rangle$ is independent in $(N_0,N_2)$.
\end{proof}

%%%%%%%%%%%%%%%%%%%%%%%%%%%%%%%%%%
%%       section           %%%%%%%
%%%%%%%%%%%%%%%%%%%%%%%%%%%%%%%%%%
\section{The Relation $\preceq^{NF}_{\lambda^+}$}

The main point in Shelah's approach is to construct a relation, $\preceq^{NF}_{\lambda^+}$, on $K_{\lambda^+}$, such that $(K_{\lambda^+},\preceq^{NF}_{\lambda^+})$ satisfies the amalgamation property. The construction of $\preceq^{NF}_{\lambda^+}$ is done by the following steps:
\begin{enumerate}
\item Assume that the class of uniqueness triples, $K^{3,uq}$, satisfies the existence property, \item use $K^{3,uq}$ to construct a non-forking relation, $NF$, on quadruples of models of cardinality $\lambda$, \item use $NF$ to construct a non-forking relation, $\widehat{NF}$ on quadruples of models: two of cardinality $\lambda$ and two of cardinality $\lambda^+$, \item use $\widehat{NF}$ to construct a binary relation, $\preceq^{NF}_{\lambda^+}$ on $K_{\lambda^+}$.
\end{enumerate}

In this section, we define $\preceq^{NF}_{\lambda^+}$ and present its basic properties.  
%So while in \cite{shh}.II, \cite{jrsh875}, \cite{jrtame} and \cite{bo3}, 
%\cite{bose} I have no time to add it now 
%the goal is to derive a good non-forking $\lambda^+$-frame, here, our main goal is to avoid the changing of the %relation $\preceq \restriction K_{\lambda^+}$. 
   
In Definition \ref{definition of NF}, we list the axioms for a relation $NF$ for a model of size $\lambda$ is independent from a model of size $\lambda$ over a model of size $\lambda$ in a model of size $\lambda$. We  denote `the relation $NF$ satisfies the list of the axioms' by $\bigotimes_{NF}$. Fact \ref{if the uniqueness triples satisfy the existence property then NF} presents sufficient conditions for the existence of a relation $NF$ for which $\bigotimes_{NF}$ holds and respecting $\frak{s}$.
  
\begin{definition}\label{definition of non-forking relation on models}\label{definition of NF}
Let $NF \subseteq \ ^4K_\lambda$. \emph{$\bigotimes_{NF}$} means that the following hold:
\begin{enumerate}[(a)]
\item If $NF(M_0,M_1,M_2,M_3)$ then for each $n\in \{1,2\}$
$M_0\leq M_n\leq M_3$ and $M_1 \cap M_2=M_0$. \item Monotonicity:
if $NF(M_0,M_1,M_2,M_3)$, $N_0=M_0$ and for each $n<3$ $N_n\leq
M_n\wedge N_0\leq N_n\leq N_3, (\exists N^{*})[M_3\leq N^{*}\wedge
N_3\leq N^{*}]$ then $NF(N_0 \allowbreak ,N_1,N_2,N_3)$. \item
Extension: For every $N_0,N_1,N_2 \in K_\lambda$, if for each
$l\in \{1,2\}$ $N_0 \leq N_l$ and $N_1\bigcap \allowbreak
N_2=N_0$, then for some $N_3 \in K_\lambda$,
$NF(N_0,N_1,N_2,N_3)$. \item Weak Uniqueness: Suppose for $x=a,b$,
$NF(N_0,N_1,N_2,N^{x}_3)$. Then there is a joint embedding of
$N^a,N^b \ over \ N_1 \bigcup N_2$. \item  Symmetry: For every
$N_0,N_1,N_2,N_3 \in K_\lambda$, $NF(N_0,N_1,N_2,N_3)
\Leftrightarrow NF(N_0,N_2,N_1,N_3)$. \item Long transitivity: For
$x=a,b$, let $\langle  M_{x,i}:i\leq \alpha^* \rangle$ an
increasing continuous sequence of models in $K_\lambda$. Suppose
that for each $i<\alpha^*$, $NF(M_{a,i},\allowbreak
M_{a,i+1},M_{b,i},\allowbreak M_{b,i+1})$. Then
$NF(M_{a,0},M_{a,\alpha^{*}},M_{b,0},M_{b,\alpha^{*}})$. \item $NF$ is closed under isomorphisms: if $NF(M_0,M_1,M_2,M_3)$ and
$f:M_3 \to N_3$ is an isomorphism then
$NF(f[M_0],f[M_1],f[M_2],f[M_3])$.
\end{enumerate}
\end{definition}

The next two definitions are needed in order to state Fact \ref{if the uniqueness triples satisfy the existence property then NF}.

\begin{definition}\label{NF respects the frame}
Let $NF$ be a relation such that $\bigotimes_{NF}$ holds. The relation $NF$ \emph{respects the frame} $\frak{s}$ means that if $NF(M_0,M_1,M_2,M_3)$, $a \in M_1-M_0$ and $tp(a,M_0,M_1)$ is basic then $tp(a,M_2,M_3)$ does not fork over $M_0$.
\end{definition}

\begin{definition}
Let $\frak{s}$ be a semi-good non-forking $\lambda$-frame. $\frak{s}$ is said to satisfy the \emph{conjugation} property when $K$ is categorical in $\lambda$ in the following strong sense: for every $M_1$ and $M_2$ of cardinality $\lambda$ and types $p_1 \in S^{bs}(M_1)$ and $p_2 \in S^{bs}(M_2)$, if $M_1 \preceq M_2$ and $p_2$ is the non-forking extension of $M_1$ then there is an isomorphism $f:M_1 \to M_2$ such that $f(p_1)=p_2$. 
\end{definition}

By \cite[Theorem 5.5.4]{jrsh875}
(and \cite[Definitions 5.2.1,5.2.6]{jrsh875}):
\begin{fact} \label{jrsh875.5.15}\label{if the uniqueness triples satisfy the existence property then NF}
If the class of uniqueness triples satisfies the existence property and $\frak{s}$ satisfies the conjugation property (see \cite[Definition 2.5.5]{jrsh875}) then there is a (unique) relation $NF \subseteq \ ^4K_\lambda$ for which $\bigotimes_{NF}$ holds and $NF$ respects the frame $\frak{s}$.
\end{fact}

%We recall that by \cite{complete...}, every good frame satisfies the conjugation property.

From now on we assume:
\begin{hypothesis}\label{hypothesis for bar{NF}}
\mbox{}
\begin{enumerate}
\item $K$ is categorical in $\lambda$,
\item $\frak{s}$ is a semi-good non-forking $\lambda$-frame, \item $\frak{s}$ satisfies the conjugation property and \item the class of uniqueness triples satisfies the existence property.
\end{enumerate}
\end{hypothesis}

In Definition \ref{definition of widehat{NF}}, we use the relation $NF$ to present a relation for: a model of size $\lambda$ is independent from a model of size $\lambda^+$ over a model of size $\lambda$ in a model of size $\lambda^+$.
\begin{definition}\label{definition of widehat{NF}}\label{5.14} Define a 4-ary relation $\widehat{NF}$ on $K$ by $$\widehat{NF}(N_0,N_1,M_0,\allowbreak M_1)$$ when the following hold:
\begin{enumerate}
\item $N_0,N_1$ are of cardinality $\lambda$, \item $M_0,M_1$ are of cardinality $\lambda^+$,
\item There are filtrations $\langle
N_{0,\alpha}:\alpha<\lambda^+ \rangle,\ \langle
N_{1,\alpha}:\alpha<\lambda^+ \rangle$ of $M_0,M_1$ respectively, such that $N_{0,0}=N_0$, $N_{1,0}=N_1$ and
for every $\alpha<\lambda^+$ we have $NF(N_{0,\alpha},N_{1,\alpha},N_{0,\alpha+1},N_{1,\alpha+1})$.
\end{enumerate}
\end{definition}

By \cite[Theorem  6.1.3]{jrsh875}:
\begin{fact} [basic properties of $\widehat{NF}$]
\label{5.15}\label{the widehat{NF}-properties} \mbox{}
\begin{enumerate}[(a)] \item Disjointness: If
$\widehat{NF}(N_0,N_1,M_0,M_1)$ then $N_1 \bigcap M_0=N_0$. \item
Monotonicity: Suppose $\widehat{NF}(N_0,N_1,M_0,M_1),\ N_0 \preceq
N^{*}_1 \preceq N_1,\ N_1^* \bigcup M_0 \allowbreak \subseteq
M^*_1 \preceq M_1$ and $M_1^* \in K_{\lambda^+}$. Then
$\widehat{NF}(N_0,N^{*}_1,M_0,M^*_1)$. \item Extension: Suppose
$n<2 \Rightarrow N_n \in K_\lambda,\ M_0 \in K_{\lambda^+},\ N_0
\preceq N_1,\ N_0 \preceq M_0,\ N_1 \bigcap M_0=N_0$. Then there
is a model $M_1$ such that $\widehat{NF}(N_0,N_1,\allowbreak
M_0,M_1)$. \item Weak Uniqueness: If $n<2 \Rightarrow
\widehat{NF}(N_0,N_1,M_0,M_{1,n})$, then there are $M,f_0,f_1$
such that $f_n$ is an embedding of $M_{1,n}$ into $M$ over $N_1
\bigcup M_0$. \item Respecting the frame: Suppose
%$\frak{s}$ satisfies the conjugation property,
$\widehat{NF}(M_0,M_1,N_0,N_1)$ and $tp(a,M_0,N_0) \in \allowbreak S^
{bs} \allowbreak (M_0)$. Then $tp(a,M_1,N_1)$ does not fork over
$M_0$.
\end{enumerate}
\end{fact}

Proposition \ref{a new version of widehat{NF} respects the frame} says roughly that '$\widehat{NF}$ respects the frame $\frak{s}$'. But Proposition \ref{a new version of widehat{NF} respects the frame} and Fact \ref{the widehat{NF}-properties}(e) are different (one might say that `there is a symmetry between them'). 
%Notice that the cardinality of $N_0$ is $\lambda^+$ (but the cardinality of $M_0$ is $\lambda$).
In Proposition \ref{widehat{NF} respects independence}, we generalize Proposition \ref{a new version of widehat{NF} respects the frame}, proving that the relation $\widehat{NF}$ respects independence.
\begin{proposition}\label{a new version of widehat{NF} respects the frame}\label{a new version of respecting the frame}\label{second version of respecting the frame}
Suppose:
\begin{enumerate}
\item $M_0,M_1 \in K_\lambda$, \item $N_0,N_1 \in K_{\lambda^+}$, \item $\widehat{NF}(M_0,M_1,N_0,N_1)$, \item $a \in M_1-M_0$, \item $tp(a,M_0,M_1)$ is basic.
\end{enumerate}
Then $tp(a,N_0,N_1)$ does not fork over $M_0$. 
\end{proposition}

\begin{proof}
Let $M' \in K_\lambda$ with $M_0 \preceq M' \preceq N_0$. We should prove that $tp(a,M',N_1)$ does not fork over $M_0$.  By the definition of $\widehat{NF}$, there are filtrations $\langle M_{0,\alpha}:\alpha<\lambda^+ \rangle$ and $\langle M_{1,\alpha}:\alpha<\lambda^+ \rangle$ of $N_0$ and $N_1$ respectively, such that $M_{0,0}=M_0$, $M_{1,0}=M_1$ and $NF(M_{0,\alpha},M_{0,\alpha+1},M_{1,\alpha},M_{1,\alpha+1})$ holds for each $\alpha<\lambda^+$. Take $\alpha<\lambda^+$ such that $M' \subseteq M_{0,\alpha}$. Since $M' \preceq N_0$ and $M_{0,\alpha} \preceq N_0$ we have $M' \preceq M_{0,\alpha}$. By long transitivity (Definition \ref{definition of NF})(f), $NF(M_{0,0},M_{0,\alpha},M_{1,0},M_{1,\alpha})$, namely, $NF(M_0,M_{0,\alpha},M_1,M_{1,\alpha})$. So by monotonicity of $NF$ (Definition \ref{definition of NF}(b)),  $NF(M_0,M',M_1,M_{1,\alpha})$. Since the relation $NF$ respects the frame $\frak{s}$, it yields $tp(a,M',N_1)$ does not fork over $M_0$.
\end{proof}

We now define a binary relation $\preceq^{NF}_{\lambda^+}$ on
$K_{\lambda^+}$, that is based on the relation $\widehat{NF}$:
\begin{definition} \label{6.1}\label{6.4 in april}\label{definition of preceq^{NF}}
Suppose $M_0,M_1 \in K_{\lambda^+}$, $M_0 \preceq M_1$. Then $M_0
\preceq^{NF}_{\lambda^+} M_1$ means that $\widehat{NF}(N_0,N_1,M_0,M_1)$ for some $N_0,N_1 \in
K_\lambda$.
\end{definition}

By \cite[Proposition 6.1.6]{jrsh875} and \cite[Theorem 7.1.18(a)]{jrsh875}:
\begin{fact} \label{6.2}\label{6.5}\label{preceq^{NF}-properties}\label{basic properties of preceq^{NF}_{lambda^+}}
$(K_{\lambda^+},\preceq^{NF}_{\lambda^+})$ satisfies the following
properties:
\begin{enumerate}[(a)]
\item Suppose $M_0 \preceq M_1,\ n<2 \Rightarrow M_n \in
K_{\lambda^+}$. For $n<2$, let $\langle
N_{n,\epsilon}:\epsilon<\lambda^+ \rangle$ be a representation of
$M_n$. Then $M_0 \preceq^{NF}_{\lambda^+}M_1$ iff there is a club
$E \subseteq \lambda^+$ such that $(\epsilon<\zeta \wedge
\{\epsilon,\zeta\} \subseteq E) \Rightarrow
NF(N_{0,\epsilon},N_{0,\zeta},N_{1,\epsilon},N_{1,\zeta})$. \item
$\preceq^{NF}_{\lambda^+}$ is a partial order. \item If $M_0
\preceq M_1 \preceq M_2$ and $M_0 \preceq^{NF}_{\lambda^+} M_2$
then $M_0 \preceq^{NF}_{\lambda^+} M_1$. \item
%$(K_{\lambda^+},\preceq^{NF}_{\lambda^+})$ satisfies Axiom c of
%AEC in $\lambda^+$, i.e.: 
If $\delta \in \lambda^{+2}$ is a limit
ordinal and $\langle M_\alpha:\alpha<\delta \rangle$ is a
$\preceq^{NF}_{\lambda^+}$-increasing continuous sequence, then
$M_0 \preceq^{NF}_{\lambda^+} \bigcup_{\alpha<\delta} M_\alpha$. Moreover, if $M_\alpha \in K^{sat}$ for each $\alpha<\delta$ then $\bigcup_{\alpha<\delta}M_\alpha \in K^{sat}$. \item $K_{\lambda^+}$ has
no $\preceq^{NF}_{\lambda^+}$-maximal model. 
%\item LST for
%$\preceq^{NF}_{\lambda^+}$: If $M_0 %\preceq^{NF}_{\lambda^+} M_1,\
%n<2 \Rightarrow (A_n \subseteq M_n \wedge |A_n| %\leq \lambda)$,
%then there are models $N_0,N_1 \in K_\lambda$ such %that:
%$\widehat{NF}(N_0,N_1,M_0,M_1)$ and $n<2 %\Rightarrow A_n\subseteq
%N_n$.
\end{enumerate}
\end{fact}

\begin{proposition}\label{remark preceq^{NF}}
Let $M_1,M_2$ be models of cardinality $\lambda^+$ with $M_1 \preceq M_2$. Then $M_1 \preceq^{NF}_{\lambda^+}M_2$ if and only if for every two filtrations $\langle M_{1,\alpha}:\alpha<\lambda^+ \rangle$ and $\langle M_{2,\alpha}:\alpha<\lambda^+ \rangle$ of $M_1$ and $M_2$ respectively, for some club $E$ of $\lambda^+$ for every $\alpha \in E$ we have $\widehat{NF}(M_{1,\alpha},M_{2,\alpha},M_1,M_2)$. 
\end{proposition}

\begin{proof}
 
On the one hand, suppose $M_1 \preceq^{NF}_{\lambda^+}M_2$. Let $\langle M_{1,\alpha}:\alpha<\lambda^+ \rangle$ and $\langle M_{2,\alpha}:\alpha<\lambda^+ \rangle$ be two filtrations of $M_1$ and $M_2$ respectively. By Fact \ref{preceq^{NF}-properties}(a), for some club $E$ of $\lambda^+$, for every $\epsilon,\zeta \in E$ if $\epsilon<\zeta$ then $NF(M_{1,\epsilon},M_{1,\zeta},M_{2,\epsilon},M_{2,\zeta})$. Let $\alpha \in E$. Then the filtrations $\langle M_{1,\epsilon}:\epsilon \in E-\alpha \rangle$ and $\langle M_{2,\epsilon}:\epsilon \in E-\alpha \rangle$ wittness that $\widehat{NF}(M_{1,\alpha},M_{2,\alpha},M_1,M_2)$. 

Conversely, we have to prove that $M_1 \preceq^{NF}_{\lambda^+} M_2$. Take filtrations $\langle M_{1,\alpha}:\alpha<\lambda^+ \rangle$ and $\langle M_{2,\alpha}:\alpha<\lambda^+ \rangle$ of $M_1$ and $M_2$ respectively. By assumption, for some club $E$ of $\lambda^+$ for every $\alpha \in E$, we have $\widehat{NF}(M_{1,\alpha},M_{2,\alpha},M_1,M_2)$. Take $\alpha \in E$. Since $\widehat{NF}(M_{1,\alpha},M_{2,\alpha},M_1,M_2)$, by Definition \ref{definition of preceq^{NF}}, $M_1 \preceq^{NF}_{\lambda^+}M_2$. 

\end{proof}

%%%%%%%%%%%%%%%%%%%%%%
%%    section     %%%%%%%%%%%%%%
%%%%%%%%%%%%%%%%%%%%%%
\section{Are the Relations $\preceq^{NF}_{\lambda^+}$ and $\preceq \restriction K_{\lambda^+}$ Equivalent?}  

The following question is open:
\begin{question}\label{equivalent?}
Let $\frak{s}$ be a semi-good non-forking $\lambda$-frame. Suppose that $(K_{\lambda^+},\preceq \restriction K_{\lambda^+})$ satisfies the amalgamation property and the class of uniqueness triples satisfies the existence property. Given $M,M^+ \in K_{\lambda^+}$. Is Statement \ref{statement equivalence} true?
\end{question}

\begin{statement}\label{statement equivalence}
$$M \preceq M^+ \text{ if and only if } M \preceq^{NF}_{\lambda^+}      M^+              $$
\end{statement}

Recall \cite[Definition 4.7]{jrprime}:
\begin{definition}
$\frak{s}$ is said to be \emph{successful semi-good$^+$} when Statement \ref{statement equivalence} holds for every $M,M^+ \in K^{sat}$.
\end{definition}
 
Here, in Theorem \ref{we can use NF with tameness for saturated models}, we prove that $\frak{s}$ is successful semi-good$^+$ assuming amalgamation in $\lambda^+$ and $(\lambda,\lambda^+)$-tameness for non-forking types over saturated models. In Theorem \ref{we can use NF}, we prove Statement \ref{statement equivalence}, for every $M,M^+ \in K_{\lambda^+}$, under a stronger assumption.

%In Theorems \ref{we can use NF} and \ref{we can use NF with tameness for saturated models} we conclude Statement \ref{statement equivalence} in some sense.  

%In Corollaries \ref{corollary 3} and \ref{corollary 4} we obtain the equivalence between the amalgamation property in $\lambda^+$ and Statement \ref{statement equivalence}.

Before stating Theorem \ref{we can use NF}, we make several preparations. Fact \ref{fact density of basic types over models of greater cardinality} is a restatement of \cite[Theorem 2.6.8.a]{jrsh875}.
\begin{fact}\label{fact density of basic types over models of greater cardinality}
If $N,N^+$ are models of cardinality $\lambda$ at least with $N \prec N^+$ then there is an element $a \in N^+-N$ such that $tp(a,N,N^+)$ is basic. 
\end{fact}

In Definition \ref{our pairs}, we define a partial order $(A,<_A)$, which playing an important role in the proof of Theorem \ref{we can use NF}. In Propositions \ref{importance of maximal} and \ref{<_A satisfies axiom c}, we present two properties of this partial order.

\begin{definition}\label{our pairs}
Define $A$ as the class of pairs, $(M_1,M^+_1)$ of models of cardinality $\lambda^+$ with $M_1 \preceq M^+_1$. 
Define a strict partial order, $<_A$
on $A$, by: $(M_1,M^+_1) <_A (M_2,M^+_2)$ when the following hold:
\begin{enumerate}
\item $M_1 \preceq^{NF}_{\lambda^+} M_2$, 
\item $M^+_1 \preceq M^+_2$,
\item $M_2 \cap M^+_1 \neq M_1$.
\end{enumerate}

\begin{displaymath}
\xymatrix{M_2 \ar[r]^{id} & M^+_2 \\ 
M_1 \ar[r]^{id} \ar[u]^{\preceq^{NF}_{\lambda^+}} & M^+_1 \ar[u]^{id}  
}
\end{displaymath}

\end{definition}

\begin{proposition}\label{importance of maximal}
Assume that $(K,\preceq)$ satisfies the strong $(\lambda,\lambda^+)$-tameness for non-forking types property. Let $(N,N^+) \in A$. If $(N,N^+)$ is $<_A$-maximal then $N=N^+$.
\end{proposition}

\begin{proof}
Let $(N,N^+)$ be a pair in $A$ with $N \neq N^+$. We should prove that $(N,N^+)$ is not $<_A$-maximal. By Fact \ref{fact density of basic types over models of greater cardinality} (density of basic types over models of cardinality greater than $\lambda$), for some $a \in N^+-N$ the type $tp(a,N,N^+)$ is basic. So there is $N^- \in K_\lambda$ such that $N^- \preceq N$ and $tp(a,N,N^+)$ does not fork over $N^-$. For some $N^-_1 \in K_\lambda$ and some $b \in N^-_1$ we have $tp(b,N^-,N^-_1)=tp(a,N^-,N^+)$.
\begin{displaymath}
\xymatrix{b \in N^-_1 \ar[rr]^{f} && N_1 \ar[r]^{g} & N_1^+ \\ 
N^- \ar[rr]^{id}  \ar[u]^{id} && N \ar[r]^{id} \ar[u]^{\preceq^{NF}_{\lambda^+}}   & N^+ \ni a \ar[u]^{id}
}
\end{displaymath}
By Fact \ref{the widehat{NF}-properties}(c), for some amalgamation $(id \restriction N,f,N_1)$ of $N$ and $N^-_1$ over $N^-$ we have $\widehat{NF}(N^-,f[N^-_1],N,N_1)$. So by Proposition \ref{second version of respecting the frame}, $tp(f(b),N,N_1)$ does not fork over $N^-$. So by the strong $(\lambda,\lambda^+)$-tameness for non-forking types property, there is an amalgamation $(id \restriction N^+,g,N_1^+)$ of $N^+$ and $N_1$ over $N$ with $g(f(b))=a$. We have $$(N,N^+)<_A(g[N_1],N_1^+)$$ (because $N^+ \preceq N_1^+$, $N \preceq^{NF}_{\lambda^+} g[N_1]$ and $a \in g[N_1] \cap N^+-N$). 
\end{proof}

\begin{proposition}\label{<_A satisfies axiom c}
If $\langle (M_\alpha,M^+_\alpha):\alpha<\delta \rangle$ is a $<_A$-increasing continuous sequence of pairs in $A$ then $$(M_\beta,M^+_\beta)<_A(\bigcup_{\alpha<\delta}M_\alpha,\bigcup_{\alpha<\delta}M^+_\alpha)$$ for each $\beta<\delta$.
\end{proposition}

\begin{proof}
Without loss of generality, $\beta=0$. We should prove that $(M_0,M^+_0)<_A(\bigcup_{\alpha<\delta}M_\alpha,\bigcup_{\alpha<\delta}M^+_\alpha)$.
By smoothness (one of the AEC's axioms), $\bigcup_{\alpha<\delta}M_\alpha \preceq \bigcup_{\alpha<\delta}M^+_\alpha$, so $(\bigcup_{\alpha<\delta}M_\alpha,\bigcup_{\alpha<\delta}M^+_\alpha) \in A$. By Fact \ref{basic properties of preceq^{NF}_{lambda^+}}(d), $M_0 \preceq^{NF}_{\lambda^+} \bigcup_{\alpha<\delta}M_\alpha$. By the definition of AEC, $M^+_0 \preceq \bigcup_{\alpha<\delta}M^+_\alpha$.
\end{proof}

\begin{comment}
%I delete it only because I don't understand it now.
\begin{question}
Is it possible to eliminate the third assumption in Theorem \ref{we can use NF}?
\end{question}
\end{comment}

\begin{theorem}\label{we can use NF}
Suppose:
\begin{enumerate}
\item Hypothesis \ref{hypothesis for bar{NF}},
\item $(K,\preceq)$ satisfies the amalgamation in $\lambda^+$ property and \item $(K,\preceq)$ satisfies the $(\lambda,\lambda^+)$-tameness for non-forking types property.
\end{enumerate}
Then for every two models $M,M^+$ of cardinality $\lambda^+$ the following holds: $$M \preceq M^+ \Leftrightarrow M \preceq^{NF}_{\lambda^+}M^+.$$
\end{theorem}

\begin{proof}
If $M \preceq^{NF}_{\lambda^+}M^+$ then by definition $M \preceq M^+$. 

Conversely, suppose $M \preceq M^+$. Without loss of generality, $M \neq M^+$.
 
It is sufficient to find a pair $(N,N^+) \in A$ such that $(M,M^+)<_A(N,N^+)$ and $N=N^+$, because it yields $M^+ \preceq N$ and $M \preceq^{NF}_{\lambda^+}N$ and so by Fact \ref{preceq^{NF}-properties}(c), $M \preceq^{NF}_{\lambda^+}M^+$. By Proposition \ref{strong iff amalgamation (for non-forking types)}, the strong $(\lambda,\lambda^+)$-tameness for non-forking types property holds. Hence, by Proposition \ref{importance of maximal}, it is sufficient to find a pair $(N,N^+) \in A$ such that $(M,M^+)<_A(N,N^+)$ and $(N,N^+)$ is a  $<_A$-maximal pair in $A$. 
  
For the sake of a contradiction, assume that there is no $<_A$-maximal pair. We choose by induction on $\alpha<\lambda^{++}$ a pair $(M_\alpha,M^+_\alpha) \in A$ such that for every $\alpha<\lambda^{++}$, $(M_\alpha,M^+_\alpha) <_A (M_{\alpha+1},M^+_{\alpha+1})$ and for every limit $\alpha<\lambda^{++}$, $M_\alpha=\bigcup_{\beta<\alpha}M_\beta$ and $M^+_\alpha=\bigcup_{\beta<\alpha}M^+_\beta$ (so by Proposition \ref{<_A satisfies axiom c}, $(M_\beta,M^+_\beta)<_A(M_\alpha,M^+_\alpha)$ for each $\beta<\alpha$). Define $M_{\lambda^{++}}:=\bigcup_{\alpha<\lambda^{++}}M_\alpha$. The sequences $\langle M_\alpha:\alpha<\lambda^{++} \rangle$ and $\langle M^+_\alpha \cap M_{\lambda^{++}}:\alpha<\lambda^{++} \rangle$ are filtrations of $M_{\lambda^{++}}$. So for some $\alpha<\lambda^{++}$ (actually, for a club of $\alpha$'s) we have $M_\alpha=M^+_\alpha \cap M_{\lambda^{++}}$. So $$ M_\alpha \subseteq M^+_\alpha \cap M_{\alpha+1} \subseteq M^+_\alpha \cap M_{\lambda^{++}}=M_\alpha.$$
Therefore $M^+_\alpha \cap M_{\alpha+1}=M_\alpha$, which is impossible, because $(M_\alpha,M^+_\alpha) <_A (M_{\alpha+1},M^+_{\alpha+1})$. A contradiction. 
\begin{comment}
This contradiction shows that there is a $<_A$-extension $(N,N^+)$ of $(M,M^+)$, such that $(N,N^+)$ is $<_A$-maximal. So by Proposition \ref{importance of maximal}, $N=N^+$. So $M \preceq^{NF}_{\lambda^+}N$ and $M^+ \preceq N^+=N$. But we assumed that $M \preceq M^+$. Hence, by Fact \ref{preceq^{NF}-properties}(c), $M \preceq^{NF}_{\lambda^+}M^+$. Theorem \ref{we can use NF} is proved.
\end{comment}
\end{proof}

Hypotheses (2) and (3) in Theorem \ref{we can use NF} relate to all the models of cardinality $\lambda^+$. Theorem \ref{we can use NF with tameness for saturated models} is one of the main theorems of the paper. It is a version of Theorem \ref{we can use NF}, where Hypotheses (2) and (3) relate to the saturated models (in $\lambda^+$ over $\lambda$) only. Before stating it, we make preparations.

The proof of Theorem \ref{we can use NF with tameness for saturated models} is similar to the  proof of Theorem \ref{we can use NF}, but more complicated. The main difficulty is in the proof of Proposition \ref{importance of maximal for saturated models}, the analogous of Proposition \ref{importance of maximal}. In the proof of Proposition \ref{importance of maximal}, we use $(\lambda,\lambda^+)$-tameness for non-forking types. But in order to apply the $(\lambda,\lambda^+)$-tameness for non-forking types over saturated models property, we should prove that the model $N$, appearing in the proof of Proposition \ref{importance of maximal}, is saturated in $\lambda^+$ over $\lambda$.

In order to overcome this difficulty, we replace the class of pairs $A$, by the class $B$ of pairs $(M_1,M_1^+) \in A$ such that  $M_1 \in K^{sat}$. Now we can apply the $(\lambda,\lambda^+)$-tameness for non-forking types over saturated models property. 

Unfortunately, a new problem arises: Not every $<_A$ extension of a pair in $B$ is in $B$. In order to solve this problem, we use the relation $\prec^+_{\lambda^+}$ and the $\prec^+_{\lambda^+}$-game. The relation $\prec^+_{\lambda^+}$ is defined in \cite[Definition 7.1.4]{jrsh875}, but only several properties of $\prec^+_{\lambda^+}$ are applied here, not its precise definition. The following fact exhibits the properties of $\prec^+_{\lambda^+}$, that are applied in the proof of Theorem \ref{we can use NF with tameness for saturated models}.
\begin{fact}\label{prec^+-properties}
There is a relation $\prec^+_{\lambda^+}$ on $K_{\lambda^+}$, satisfying the following properties:
\begin{enumerate}
\item for every $N_1 \in K_{\lambda^+}$ we can find $N_2$ such that $N_1\prec^+_{\lambda^+}N_2$, \item if $N_1\prec^+_{\lambda^+}N_2$ then $N_1 \preceq^{NF}_{\lambda^+} N_2$ and $N_2$ is saturated in $\lambda^+$ over $\lambda$ and \item Player 1 has a winning strategy in the $\prec^+_{\lambda^+}$-game (see Definition \ref{definition of the game} below).
\end{enumerate}
\end{fact}

\begin{proof}
Clauses (1)-(3) of Fact \ref{prec^+-properties} are restatements of Theorems 7.1.12(a), 7.1.10(a),(b) and 7.1.12(c) of \cite{jrsh875} respectively.
\end{proof}

We restate \cite[Definition 7.1.11]{jrsh875} as follows:
\begin{definition}\label{definition of the game}
The \emph{$\prec^+_{\lambda^+}$-game} is a game between two players, Player 0 and Player 1. The game has $\lambda^+$ rounds. In any round, the players
choose two models, $N_{0,\alpha},N_{1,\alpha}$ in $K_\lambda$ (usually, Player 0 chooses the model $N_{0,\alpha}$ and Player 1 chooses the model $N_{1,\alpha}$, but in the first round, Player 0 chooses both) with the following rules:

The first round: Player 0 chooses models $N_{0,0},N_{1,0} \in
K_\lambda$ with $N_{0,0} \preceq N_{1,0}$ and Player 1 does not do
anything.

The $\alpha$ round where $\alpha$ is limit: Player 0 must choose
$N_{0,\alpha}:=\bigcup_{\beta<\alpha} N_{0,\beta}$ and Player 1
must choose $N_{1,\alpha}:=\bigcup_{\beta<\alpha} N_{1,\beta}$.

The $\alpha+1$ round: Player 0 chooses a model $N_{0,\alpha+1}$
such that the following hold:
\begin{enumerate}
\item $N_{0,\alpha} \preceq N_{0,\alpha+1}$. \item $N_{0,\alpha+1}
\bigcap N_{1,\alpha}=N_{0,\alpha}$.
\end{enumerate}

After Player 0 chooses $N_{0,\alpha+1}$, Player 1 has to choose
$N_{1,\alpha+1}$ such that $NF(N_{0,\alpha},N_{1,\alpha},N_{0,\alpha+1},N_{1,\alpha+1})$.

At the end of the play, Player 1 wins the game if $$\bigcup_{\alpha<\lambda^+}
N_{0,\alpha} \prec^+_{\lambda^+} \bigcup_{\alpha<\lambda^+}
N_{1,\alpha}.$$ Otherwise Player 0 wins the game.

\emph{A position after $\alpha+\frac{1}{2}$ rounds} is a triple
$$(\alpha,\langle
N_{0,\beta}:\beta \leq \alpha+1 \rangle,\langle N_{1,\beta}:\beta
\leq \alpha \rangle)$$ that satisfies the following conditions:
\begin{enumerate}
\item $\alpha<\lambda^+$. \item $\langle N_{0,\beta}:\beta \leq
\alpha+1 \rangle,\ \langle N_{1,\beta}:\beta \leq \alpha \rangle$
are increasing continuous sequences of models in $K_\lambda$.
\item
$NF(N_{0,\beta},N_{1,\beta},N_{0,\beta+1},N_{1,\beta+1})$ for
$\beta<\alpha$. \item $N_{0,\alpha+1} \bigcap
N_{1,\alpha}=N_{0,\alpha}$.
\end{enumerate}

\emph{A strategy for Player 1} is a function $F$ that assigns a
model $N_{1,\alpha+1}$ satisfying $NF(N_{0,\alpha},N_{1,\alpha},N_{0,\alpha+1},N_{1,\alpha+1})$ to each position after $\alpha+\frac{1}{2}$ rounds.
 
\emph{A winning strategy for Player 1} is a
strategy for Player 1, such that if Player 1 acts by it, then he
wins the game, no matter what Player 0 does.
\end{definition}

We define a partial order $(B,<_B)$ such that $B$ is the class of pairs $(M_1,M_1^+) \in A$ with $M_1 \in K^{sat}$ and $<_B$ is the restriction of $<_A$ to $B$.
\begin{definition}\label{our pairs where the first is saturated}
Define $B$ as the class of pairs, $(M_1,M^+_1)$ of models of cardinality $\lambda^+$ such that $M_1$ is saturated over $\lambda$ and $M_1 \preceq M^+_1$. 
Define a strict partial order, $<_B$
on $B$, by: $(M_1,M^+_1) <_B (M_2,M^+_2)$ when the following hold:
\begin{enumerate}
\item $M_1 \preceq^{NF}_{\lambda^+} M_2$, 
\item $M^+_1 \preceq M^+_2$,
\item $M_2 \cap M^+_1 \neq M_1$.
\end{enumerate}
\end{definition}

Proposition \ref{importance of maximal for saturated models} is the analog of Proposition \ref{importance of maximal}.
\begin{proposition}\label{importance of maximal for saturated models}
Assume the strong $(\lambda,\lambda^+)$-tameness for non-forking types over saturated models property. Let $(N,N^+) \in B$. If $(N,N^+)$ is $<_B$-maximal then $N=N^+$.
\end{proposition}

\begin{proof}
Let $(N,N^+)$ be a pair in $B$ with $N \neq N^+$. We should prove that $(N,N^+)$ is not $<_B$-maximal.
\begin{displaymath}
\xymatrix{b \in N_{1,0} \ar[rr]^{id} && N_{1,\alpha} \ar[r]^{id} & N_{1,\alpha+1} \ar[rr]^{id} && N_1 \ar[r]^{g} & N_1^+ \\ 
N_{0,0} \ar[rr]^{id}  \ar[u]^{f_0=id} && N_{0,\alpha} \ar[r]^{id} \ar[u]^{f_{\alpha}}  & N_{0,\alpha+1} \ar[rr]^{id} \ar[u]^{f_{\alpha+1}} && N \ar[r]^{id} \ar[u]^{f_{\lambda^+}} \ar[u]_{\prec^+_{\lambda^+}}   & N^+ \ni a \ar[u]_{f^*}
}
\end{displaymath}
By Fact \ref{fact density of basic types over models of greater cardinality} (density of basic types over models of cardinality greater than $\lambda$), there is an element $a \in N^+-N$ such that $tp(a,N,N^+)$ is basic. Let $\langle N_{0,\alpha}:\alpha<\lambda^+ \rangle$ be a filtration of $N$. By Definition \ref{definition of basic types over models of greater cardinality}, for some $\alpha<\lambda^+$, $tp(a,N,N^+)$ does not fork over $N_{0,\alpha}$. So by renaming, without loss of generality, $tp(a,N,N^+)$ does not fork over $N_{0,0}$. Let $N_{1,0} \in K_\lambda$ and let $b \in N_{1,0}$ such that $tp(b,N_{0,0},N_{1,0})=tp(a,N_{0,0},N^+)$. 

Define $f_0:N_{0,0} \to N_{1,0}$ by $f_0(x)=x$.
Let $F$ be a winning strategy for Player 1 in the $\prec^+_{\lambda^+}$-game. We choose $N_{1,\alpha} \in K_\lambda$ and an injection $f_\alpha:N_{0,\alpha} \to N_{1,\alpha}$ by induction on $\alpha \in (0,\lambda^+)$ such that the following hold: 
\begin{enumerate}
\item $f_{\alpha+1}(x)=f_{\alpha}(x)$, for each $x \in N_{0,\alpha}$, \item $f_{\alpha+1}[N_{0,\alpha+1}] \cap N_{1,\alpha}=f_{\alpha}[N_{0,\alpha}]$, \item if $\alpha$ is limit then $f_\alpha=\bigcup_{\beta<\alpha}f_\beta$, 
%4
\item $N_{1,\alpha+1}:=F(\alpha,\langle
f_\beta[N_{0,\beta}]:\beta \leq \alpha+1 \rangle,\langle N_{1,\beta}:\beta
\leq \alpha \rangle)$, \item if $\alpha$ is limit then $N_{1,\alpha}:=\bigcup_{\beta<\alpha}N_{1,\beta}$. 
\end{enumerate}

This induction can be described as a play of the $\prec^+_{\lambda^+}$-game, where at the $\alpha$ round, Player 0 chooses the model $f_\alpha[N_{0,\alpha}]$ explicitly, by choosing $f_\alpha$ (so when we refer to the definition of the $\prec^+_{\lambda^+}$-game, $f_\alpha[N_{0,\alpha}]$ stands for $N_{0,\alpha}$). Hence, the $\alpha$ round is as follows:
For $\alpha=0$, Player 0 chooses $N_{0,0}$ and $N_{1,0}$. For $\alpha$ limit, Player 0 chooses the model  $f_\alpha[N_{0,\alpha}]=\bigcup_{\beta<\alpha}f_\beta[N_{0,\beta}]$ and Player 1 chooses the model $N_{1,\alpha}:=\bigcup_{\beta<\alpha}N_{1,\beta}$ (see Clauses (3) and (5)). 

In the $\alpha+1$ round, Player 0 chooses the model $f_{\alpha+1}[N_{0,\alpha+1}]$ (explicitly, by choosing $f_{\alpha+1}$) such that Clauses (1) and (2) hold. It is a legal move: On the one hand, since $N_{0,\alpha} \preceq N_{0,\alpha+1}$, Clause (1) yields $f_\alpha[N_{0,\alpha}]=f_{\alpha+1}[N_{0,\alpha}] \preceq f_{\alpha+1}[N_{0,\alpha+1}]$, so Condition (1) of a legal move for Player 0 holds. On the other hand, Clause (2) is Condition (2) of a legal move for Player 0. Now $N_{1,\alpha+1}:=F(\alpha,\langle
f_\beta[N_{0,\beta}]:\beta \leq \alpha+1 \rangle,\langle N_{1,\beta}:\beta \leq \alpha \rangle)$ is the choice of Player 1.  

Define $f_{\lambda^+}:=\bigcup_{\alpha<\lambda^+}f_\alpha$ and $N_1:=\bigcup_{\alpha<\lambda^+}N_{1,\alpha}$.  

Since $F$ is a winning strategy, $NF(f_\alpha[N_{0,\alpha}],f_{\alpha+1}[N_{0,\alpha+1}],N_{1,\alpha},N_{1,\alpha+1})$ holds, for every $\alpha<\lambda^+$. So $f_\alpha[N_{0,\alpha}] \preceq N_{1,\alpha}$ for each $\alpha<\lambda^+$. Since the sequences $\langle f_\alpha[N_{0,\alpha}]:\alpha<\lambda^+ \rangle$ and $\langle N_{1,\alpha}:\alpha<\lambda^+ \rangle$ are increasing  and continuous (by Clauses (3) and (5)), they witness that $\widehat{NF}(N_{0,0},N_{1,0},f_{\lambda^+}[N],N_1)$ (note that $f_{\lambda^+}[N_{0,0}]=N_{0,0}$). So by Proposition \ref{a new version of respecting the frame} $tp(b,f_{\lambda^+}[N],N_1)$ does not fork over $N_{0,0}$. 

%Let $g$ be an injection with domain $N_1$ extending $f_{\lambda^+}^{-1}$. 

Let $f^*$ be a function with domain $N^+$ extending $f_{\lambda^+}$. Then the types $tp(b,f^*[N],N_1)$ and $tp(f^*(a),f^*[N],f^*[N^+])$ do not fork over $f^*[N_{0,0}]=N_{0,0}$. But the types over $N_{0,0}$ are equal:

$$tp(b,N_{0,0},N_1)=tp(a,N_{0,0},N^+)=tp(f^*(a),N_{0,0},f^*[N^+]).$$
Since $(N,N^+) \in B$, we have $N \in K^{sat}$, so $f^*[N] \in K^{sat}$. Hence, by the strong $(\lambda,\lambda^+)$-tameness for non-forking types over saturated models property, there is an amalgamation of $N_1$ and $f^*[N^+]$ over $f^*[N]$ such that the images of $b$ and $f^*(a)$ coincide. Equivalently, for some model $N_1^+ \in K_{\lambda^+}$, for some embedding $g:N_1 \to N_1^+$ fixing $f^*[N]$ pointwise we have $g(b)=f^*(a)$ and the following diagram commutes:
\begin{displaymath} 
\xymatrix{b \in N_1 \ar[r]^{g} & N_1^+ \\
N \ar[r]^{id} \ar[u]^{f_{\lambda^+}} & N^+ \ni a \ar[u]^{f^*}
}
\end{displaymath}

Since the relation $<_B$ is closed under isomorphisms, it is sufficient to prove that $(f^*[N],f^*[N^+])<_B(g[N_1],N_1^+)$. But it follows by Clauses (2)-(5) of the following subclaim:
\begin{subclaim} 
\mbox{}\\
\begin{enumerate}
\item $f^*[N] \prec^+_{\lambda^+} N_1$, \item $f^*[N] \preceq^{NF}_{\lambda^+} g[N_1]$, \item $g[N_1] \in K^{sat}$, \item $f^*[N^+] \preceq N_1^+$ and \item $g(b)=f^*(a) \in f^*[N^+] \cap g[N_1]-f^*[N]$.
\end{enumerate}
\end{subclaim}

\begin{proof}
\mbox{}\\
\begin{enumerate}
\item Since $F$ is a winning strategy, Player 1 wins the game. Therefore $f^*[N] \prec^+_{\lambda^+}N_1$.
\item By Clause (1) and Fact \ref{prec^+-properties}(2), $f^*[N] \preceq^{NF}_{\lambda^+} N_1$. Since $g$ fixes $f^*[N]$ pointwise and the relation $\preceq^{NF}_{\lambda^+}$ is closed under isomorphism, $f^*[N]=g[f^*[N]] \preceq^{NF}_{\lambda^+} g[N_1]$.
\item By Clause (1) and Fact \ref{prec^+-properties}(2), $N_1 \in K^{sat}$. Since $K^{sat}$ is closed under isomorphisms, $g[N_1] \in K^{sat}$. \item Obvious. \item Obvious.   
\end{enumerate}
\end{proof}
 
The proof of Proposition \ref{importance of maximal for saturated models} is completed.
\end{proof}

Proposition \ref{<_B satisfies axiom c} is analogous to Proposition \ref{<_A satisfies axiom c}. 
\begin{proposition}\label{<_B satisfies axiom c}
If $\langle (M_\alpha,M^+_\alpha):\alpha<\delta \rangle$ is a $<_B$-increasing continuous sequence of pairs in $B$ then $$(M_\beta,M^+_\beta)<_B(\bigcup_{\alpha<\delta}M_\alpha,\bigcup_{\alpha<\delta}M^+_\alpha)$$ for each $\beta<\delta$.
\end{proposition}

\begin{proof}
Without loss of generality, $\beta=0$. We should prove that $(M_0,M^+_0)<_B(\bigcup_{\alpha<\delta}M_\alpha,\bigcup_{\alpha<\delta}M^+_\alpha)$. By Proposition \ref{<_A satisfies axiom c}, $(M_0,M^+_0)<_A(\bigcup_{\alpha<\delta}M_\alpha,\bigcup_{\alpha<\delta}M^+_\alpha)$. It remains to show that $\bigcup_{\alpha<\delta}M_\alpha$ is a saturated model in $\lambda^+$ over $\lambda$.

$\langle M_\alpha:\alpha<\delta \rangle$ is a $\preceq^{NF}_{\lambda^+}$-increasing and continuous sequence of models in $K^{sat}$. So by \cite[Theorem 7.18(a)]{jrsh875}, $\bigcup_{\alpha<\delta}M_\alpha \in K^{sat}$. Hence, $(\bigcup_{\alpha<\delta}M_\alpha,\bigcup_{\alpha<\delta}M^+_\alpha) \in B$. Proposition \ref{<_B satisfies axiom c} is proved.
\end{proof}

Now we prove Theorem \ref{0}:
\begin{theorem}\label{we can use NF with tameness for saturated models}\label{the main theorem of the paper}
Suppose:
\begin{enumerate}
\item Hypothesis \ref{hypothesis for bar{NF}}, \item every saturated model in $\lambda^+$ is an amalgamation base and \item $(K,\preceq)$ satisfies the $(\lambda,\lambda^+)$-tameness for non-forking types over saturated models property.
\end{enumerate}
Then for every two models $M,M^+ \in K^{sat}$ the following holds: $$M \preceq M^+ \Leftrightarrow M \preceq^{NF}_{\lambda^+}M^+.$$
\end{theorem}

\begin{proof}
Let $M,M^+$ be two saturated models in $\lambda^+$ over $\lambda$. If $M \preceq^{NF}_{\lambda^+}M^+$ then by definition $M \preceq M^+$. 

Conversely, assume that $M \preceq M^+$. As in the proof of Theorem \ref{we can use NF}, it is sufficient to find a pair $(N,N^+) \in B$ such that $(M,M^+)<_B(N,N^+)$ and $N=N^+$. By Proposition \ref{... implies strong for saturated} and Clauses (2) and (3), $(K,\preceq)$ satisfies the strong $(\lambda,\lambda^+)$-tameness for non-forking types over saturated models property. Hence, by Proposition \ref{importance of maximal for saturated models}, it is sufficient to find a pair $(N,N^+) \in B$ such that $(M,M^+)<_B(N,N^+)$ and $(N,N^+)$ is $<_B$-maximal pair in $B$.

We now can complete the proof of Theorem \ref{we can use NF with tameness for saturated models} as in the proof of Theorem \ref{we can use NF} (where Proposition \ref{importance of maximal for saturated models} replaces Proposition \ref{importance of maximal} and Proposition \ref{<_B satisfies axiom c} replaces Proposition \ref{<_A satisfies axiom c}).
\end{proof}

Corollary \ref{corollary 1} is a special case of Proposition \ref{strong iff amalgamation (for non-forking types)}. We present it, in order to show a new proof of the fact, that assuming Hypothesis \ref{hypothesis for bar{NF}}, the strong $(\lambda,\lambda^+)$-tameness for non-forking types implies the amalgamation property. 
\begin{corollary}\label{corollary 1}
Suppose:
\begin{enumerate}
\item Hypothesis \ref{hypothesis for bar{NF}} and \item $(K,\preceq)$ satisfies the strong $(\lambda,\lambda^+)$-tameness for non-forking types property.
\end{enumerate}
Then $K$ satisfies the amalgamation property in $\lambda^+$.
\end{corollary}

\begin{proof}
By \cite[Corollary 7.1.17(a)]{jrsh875}, $(K_{\lambda^+},\preceq^{NF}_{\lambda^+} \restriction K_{\lambda^+})$ satisfies the amalgamation property. In the proof of Theorem \ref{we can use NF}, we do not use the amalgamation in $\lambda^+$ property, only the strong $(\lambda,\lambda^+)$-tameness for non-forking types property. So by the proof of Theorem \ref{we can use NF}, the corollary holds.
\end{proof}

Corollary \ref{corollary 2} is a special case of Proposition \ref{strong iff amalgamation (for non-forking types over saturated models)}. It is analogous to Corollary \ref{corollary 1}. While in the proof of Corollary \ref{corollary 1}, we use the proof of Theorem \ref{we can use NF}, here, we use the proof of Theorem \ref{the main theorem of the paper}.
\begin{corollary}\label{corollary 2}\label{tameness implies every saturated model is an amalgamation base}
Suppose:
\begin{enumerate}
\item Hypothesis \ref{hypothesis for bar{NF}} and \item $(K,\preceq)$ satisfies the strong $(\lambda,\lambda^+)$-tameness for non-forking types over saturated models property.
\end{enumerate}
Then every model $M \in K^{sat}$ is an amalgamation base.
\end{corollary}

\begin{proof}
Suppose $M_0 \in K^{sat}$, $M_1,M_2 \in K_{\lambda^+}$, $M_0 \preceq M_1$ and $M_0 \preceq M_2$. We have to amalgamate $M_1,M_2$ over $M_0$. By \cite[Theorem 7.1.12.a]{jrsh875}, we can find $M_1^+,M_2^+ \in K^{sat}$ such that $M_1 \preceq M_1^+$ and $M_2 \preceq M_2^+$. Since the relation $\preceq$ is transitive, we have $M_0 \preceq M_1^+$ and $M_0 \preceq M_2^+$. Therefore by Theorem \ref{we can use NF with tameness for saturated models} (where we replace hypotheses (2) and (3) of this theorem, by the strong $(\lambda,\lambda^+)$-tameness for non-forking types over saturated models property) and Assumption (2), $M_0 \preceq^{NF}_{\lambda^+} M_1^+$ and $M_0 \preceq^{NF}_{\lambda^+} M_2^+$. So, by \cite[Theorem 7.1.18.c]{jrsh875}, we can find an amalgamation $(f_1,f_2,M_3)$ of $M_1^+,M_2^+$ over $M_0$. Hence, $(f_1 \restriction M_1,f_2 \restriction M_2,M_3)$ is an amalgamation of $M_1$ and $M_2$ over $M_0$.
\end{proof}

In Corollaries \ref{the relations are equivalent iff amalgamation holds} and \ref{the relations are equivalent iff amalgamation holds for saturated models}, we do not assume strong tameness, but tameness only.
\begin{corollary}\label{the relations are equivalent iff amalgamation holds}\label{corollary 3}
Suppose:
\begin{enumerate}
\item Hypothesis \ref{hypothesis for bar{NF}} and \item $(K,\preceq)$ satisfies the $(\lambda,\lambda^+)$-tameness for non-forking types property.
\end{enumerate}

The following conditions are equivalent:
\begin{enumerate}[(a)] 
\item the relations $\preceq^{NF}_{\lambda^+}$ and $\preceq \restriction K_{\lambda^+}$ are equivalent. 
\item the amalgamation property in $\lambda^+$ holds, \item $(K,\preceq)$ satisfies the strong $(\lambda,\lambda^+)$-tameness for non-forking types property.
\end{enumerate}
\end{corollary}

\begin{proof}
\mbox{}\\
$(a) \rightarrow (b)$: By \cite[Corollary 7.1.17(a)]{jrsh875}.\\ $(b) \rightarrow (c)$: Obvious. It is similar to Remark \ref{tameness + amalgamation implies strong tameness}.\\ $(c) \rightarrow (a)$: By Theorem \ref{we can use NF}.
\end{proof}

\begin{corollary}\label{the relations are equivalent iff amalgamation holds for saturated models}\label{corollary 4}
Suppose:
\begin{enumerate}
\item Hypothesis \ref{hypothesis for bar{NF}} and \item $(K,\preceq)$ satisfies the $(\lambda,\lambda^+)$-tameness for non-forking types over saturated models property.
\end{enumerate}

The following conditions are equivalent:
\begin{enumerate}[(a)] 
\item the relations $\preceq^{NF}_{\lambda^+} \restriction K^{nice}$ and $\preceq \restriction K^{nice}$ are equivalent. 
\item $(K^{nice},\preceq^{NF}_{\lambda^+} \restriction K^{nice})$ satisfies the amalgamation property, \item $(K,\preceq)$ satisfies the strong $(\lambda,\lambda^+)$-tameness for non-forking types over saturated models property.
\end{enumerate}
\end{corollary}

\begin{proof}
\mbox{}\\
$(a) \rightarrow (b)$: By \cite[Corollary 7.1.17(a)]{jrsh875}.\\ $(b) \rightarrow (c)$: Obvious.
\\ $(c) \rightarrow (a)$: By Theorem \ref{we can use NF with tameness for saturated models}.
\end{proof}

\section{Proving Continuity}

In this section, we get the $(\lambda,\lambda^+)$-continuity of serial independence property, using the relation $\preceq^{NF}_{\lambda^+}$.

In Proposition \ref{a rectangle with NF} we amalgamate two sequences of models such that $NF$ holds. In Proposition \ref{widehat{NF} with independence} we amalgamate models such that the relation $\widehat{NF}$ holds and in addition we get independence. Proposition \ref{widehat{NF} with independence} and the weak uniqueness property of the relation $\widehat{NF}$ yield Proposition \ref{widehat{NF} respects independence}. Proposition \ref{widehat{NF} respects independence} says that the relation $\widehat{NF}$ respects independence, in some sense. Using Proposition \ref{widehat{NF} respects independence}, we prove Theorem \ref{continuity of serial independence}. This theorem presents sufficient conditions for the $(\lambda,\lambda^+)$-continuity of serial independence property.

The proof of Proposition \ref{a rectangle of models with NF} is similar to known proofs (\cite[Proposition 3.1.10]{jrsh875}, for example), but for completeness, we give it. Proposition \ref{a rectangle with NF of width 2} is a special case of Proposition \ref{a rectangle of models with NF}. But in the proof of Proposition \ref{a rectangle of models with NF}, we apply Proposition \ref{a rectangle with NF of width 2}.
The proof of Proposition \ref{a rectangle of models with NF} is similar to known proofs (\cite[Proposition 3.1.10]{jrsh875}, for example), but for completeness, we give it. Proposition \ref{a rectangle with NF of width 2} is a special case of Proposition \ref{a rectangle of models with NF}. But in the proof of Proposition \ref{a rectangle of models with NF}, we apply Proposition \ref{a rectangle with NF of width 2}.
\begin{proposition}\label{a rectangle with NF of width 2}
Let $\alpha^* \leq \lambda^+$ and let $\langle M_\alpha:\alpha<\alpha^* \rangle$ be an increasing continuous sequence of models of cardinality $\lambda$. Let $N \in K_\lambda$ with $M_0 \preceq N$. Then we can find an increasing continuous sequence, $\langle N_\alpha:\alpha<\alpha^* \rangle$ of models in $K_\lambda$ and an embedding $f:N \to N_0$ fixing $M_0$ pointwise, such that $NF(M_\alpha,N_\alpha,M_{\alpha+1},N_{\alpha+1})$ holds for each $\alpha$ satisfying $\alpha+1<\alpha^*$.
\begin{displaymath}
\xymatrix{N \ar[r]^{f} & N_0 \ar[rr] && N_\alpha \ar[r]^{id} & N_{\alpha+1}\\
& M_0 \ar[ul]^{id} \ar[rr] \ar[u]^{id} && M_\alpha \ar[u]^{id} \ar[r]^{id} & M_{\alpha+1} \ar[u]^{id}
}
\end{displaymath} 
\end{proposition}

\begin{proof}
We choose $N_\alpha \in K_\lambda$ and an embedding $g_\alpha:M_\alpha \to N_\alpha$ by induction on $\alpha<\alpha^*$ such that for every $\alpha<\alpha^*$, the following hold:
\begin{enumerate}
\item $N \preceq N_0$, \item if $\alpha+1<\alpha^*$ then $N_\alpha \preceq N_{\alpha+1}$, $g_\alpha \subseteq g_{\alpha+1}$ and $NF(g_\alpha[M_\alpha],g_{\alpha+1}[M_{\alpha+1} \allowbreak ],N_\alpha,N_{\alpha+1})$ and \item if $\alpha$ is limit then $N_\alpha=\bigcup_{\beta<\alpha}N_\alpha$ and $g_\alpha=\bigcup_{\beta<\alpha}g_\beta$. 
\end{enumerate}

\begin{displaymath}
\xymatrix{N \ar[r]^{id} & N_0 \ar[rr]^{id} && N_\alpha \ar[r]^{id} & N_{\alpha+1}\\
& M_0 \ar[ul]^{id} \ar[rr]^{id} \ar[u]^{g_0} && M_\alpha \ar[u]^{g_\alpha} \ar[r]^{id} & M_{\alpha+1} \ar[u]^{g_{\alpha+1}}
}
\end{displaymath}
Now let $h$ be an injection of $\bigcup_{\alpha<\alpha^*}N_\alpha$ extending $(\bigcup_{\alpha<\alpha^*}g_\alpha)^{-1}$. The sequence $\langle h[N_\alpha]:\alpha<\alpha^* \rangle$ and the embedding $h \restriction N$ are as needed.
\end{proof}

%By the proof of \cite[Proposition 3.5?]{jrsh875}:
\begin{proposition} \label{a rectangle of models with NF}\label{a rectangle with NF}
Let $\alpha^*,\beta^*$ be two ordinals $\leq \lambda^+$ and let $\langle M_{a,\alpha}:\alpha<\alpha^* \rangle$ and $\langle M_{b,\alpha}:\alpha<\beta^* \rangle$ be two increasing continuous sequences of models of cardinality $\lambda$ such that $M_{a,0}=M_{b,0}$. Then there is a `rectangle of models' $\{M_{\alpha,\beta}:\alpha<\alpha^*,\beta<\beta^* \}$ and a sequence $\{f_\beta:\beta<\beta^* \}$ such that for every $\alpha<\alpha^*$ and $\beta<\beta^*$ the following hold:
\begin{enumerate}
\item $M_{\alpha,\beta} \in K_\lambda$, \item if $\alpha+1<\alpha^*$ then $M_{\alpha,\beta} \preceq M_{\alpha+1,\beta}$, \item if $\beta=\gamma+1$ then $M_{\alpha,\gamma} \preceq M_{\alpha,\gamma+1}$, \item if $\alpha$ is a limit ordinal then $M_{\alpha,\beta}=\bigcup_{\alpha'<\alpha}M_{\alpha',\beta}$, \item if $\beta$ is a limit ordinal then $M_{\alpha,\beta}=\bigcup_{\beta'<\beta}M_{\alpha,\beta'}$,  
\item $f_\beta$ is an isomorphism of $M_{b,\beta}$ onto $M_{0,\beta}$, \item if $\beta=0$ then $M_{\alpha,\beta}=M_{a,\alpha}$ and $f_\beta$ is the identity on $M_{a,0}=M_{b,0}$, \item $f_{\beta'} \subseteq f_\beta$ for every $\beta'<\beta$, \item if $\beta$ is a limit ordinal then $f_\beta=\bigcup_{\beta'<\beta}f_{\beta'}$, \item if $\beta=\gamma+1$ then $NF(M_{\alpha,\gamma},M_{\alpha,\beta},M_{\alpha+1,\gamma},M_{\beta,\gamma+1})$, unless $\alpha+1=\alpha^*$.
\end{enumerate}  
\end{proposition}

\begin{displaymath}
\xymatrix{M^a_{\alpha+1}=M_{\alpha+1,0} 
\ar[rrr]^{id} &&& M_{\alpha+1,\beta} \ar[r]^{id} & M_{\alpha+1,\beta+1} \\
M^a_{\alpha}=M_{\alpha,0} \ar[u]^{id}
\ar[rrr]^{id} &&& M_{\alpha,\beta} \ar[r]^{id} \ar[u]^{id} & M_{\alpha,\beta+1} \ar[u]^{id} \\ \\
M^a_{0}=M_{0,0} \ar[uu]^{id}
\ar[rrr]^{id} &&& M_{0,\beta} \ar[r]^{id} \ar[uu]^{id} & M_{0,\beta+1} \ar[uu]^{id} \\
M^b_0 \ar[u]^{f_0} &&& M^b_\beta \ar[u]^{f_\beta} & M^b_{\beta+1} \ar[u]^{f_{\beta+1}}
}
\end{displaymath}

\begin{proof}
For $n=1,2,..,10$, let (n)$_\beta$ denote Clause n for $\beta$.
%and let (n)$_{\alpha,\beta}$ denote Clause n for $\alpha$ and $\beta$ [so %(n)$_\beta$ is equivalent to [(n)$_{\alpha,\beta}$ for every $\alpha<\lambda^+$]] 
We choose by induction on $\beta<\beta^*$ a set of models $\{M_{\alpha,\beta}:\alpha<\alpha^*\}$ and a function $f_\beta$ such that Clauses (1)$_\beta$-(10)$_\beta$ hold. 

\emph{Case a:} $\beta=0$. In this case, Clause (7) determines the construction. Clauses (1)$_{\beta}$,(2)$_{\beta}$,(4)$_{\beta}$ and (6)$_{\beta}$ hold by an assumption. Clauses (3)$_{\beta}$,(5)$_{\beta}$,\allowbreak (8)$_{\beta}$ (9)$_{\beta}$ and (10)$_{\beta}$ are not relevant. 

\emph{Case b:} $\beta$ is a limit ordinal. In this case, Clauses (3)$_\beta$,(7)$_\beta$ and (10)$_\beta$ are not relevant. By Clause (5)$_\beta$, we must choose $M_{\alpha,\beta}=\bigcup_{\beta'<\beta}M_{\alpha,\beta'}$ and by Clause (9)$_\beta$, we must choose $f_\beta=\bigcup_{\beta'<\beta}f_{\beta'}$. Easily, Clauses (1)$_\beta$,(6)$_\beta$ and (8)$_\beta$ hold. 

Let us prove that Clause (2)$_\beta$ holds. Fix $\alpha$ with $\alpha+1<\alpha^*$. By Clause (3)$_{\beta'}$ for $\beta'<\beta$, the sequences $\langle M_{\alpha,\beta'}:\beta'<\beta \rangle$ and $\langle M_{\alpha+1,\beta'}:\beta'<\beta \rangle$ are increasing. By Clause (5)$_{\beta'}$ for $\beta'<\beta$, these sequences are continuous. So by smoothness $M_{\alpha,\beta} \preceq M_{\alpha+1,\beta}$, namely, Clause (2)$_\beta$ holds. 

It remains to prove that Clause (4)$_\beta$, namely, to show that if $\alpha$ is limit then $M_{\alpha,\beta}=\bigcup_{\alpha'<\alpha}M_{\alpha',\beta}$. On the one hand, if $x \in M_{\alpha',\beta}$ for some $\alpha'<\alpha$ then by the definition of $M_{\alpha',\beta}$, $x \in M_{\alpha',\beta'}$ for some $\beta'<\beta$. By the induction hypothesis, namely, (4)$_{\beta'}$, $M_{\alpha',\beta'} \subseteq M_{\alpha,\beta'}$. But $M_{\alpha,\beta'} \subseteq M_{\alpha,\beta}$. Whence, $\bigcup_{\alpha'<\alpha}M_{\alpha',\beta} \subseteq M_{\alpha,\beta}$. On the other hand, if $x \in M_{\alpha,\beta}$ then $x \in M_{\alpha,\beta'}$ for some $\beta'<\beta$. Therefore by (4)$_{\beta'}$, $x \in M_{\alpha',\beta'}$ for some $\alpha'<\alpha$. But $M_{\alpha',\beta'} \subseteq M_{\alpha',\beta}$. Whence, $M_{\alpha,\beta} \subseteq \bigcup_{\alpha'<\alpha}M_{\alpha',\beta}$.

\emph{Case c:} $\beta=\gamma+1$. We apply Proposition \ref{a rectangle with NF of width 2}, where $M_{\alpha,\gamma}$ stands for $M_\alpha$. By this proposition, we can find an increasing continuous sequence of models, $\langle N_\alpha:\alpha<\alpha^* \rangle$ of models in $K$ of cardinality $\lambda$ and an embedding $f:M^b_{\gamma+1} \to N_0$ such that $f_\gamma \subseteq f$ and for each $\alpha<\alpha^*$ we have $NF(M_{\alpha,\gamma},N_\alpha,M_{\alpha+1,\gamma},N_{\alpha+1})$. Define $M_{\alpha,\beta}:=N_\alpha$ and $f_\beta:=f$. Hence, we have found an increasing continuous sequence of models, $\langle M_{\alpha,\beta}:\alpha<\alpha^* \rangle$ of models in $K$ of cardinality $\lambda$ and an embedding $f_{\beta}:M^b_{\gamma+1} \to M_{0,\beta}$ such that $f_\gamma \subseteq f_\beta$ and for each $\alpha<\alpha^*$ we have $NF(M_{\alpha,\gamma},M_{\alpha,\beta},M_{\alpha+1,\gamma},M_{\alpha+1,\beta})$. So (1)$_\beta$,(2)$_\beta$, (4)$_\beta$, (6)$_\beta$, (8)$_\beta$ and (10)$_\beta$ hold. But (10)$_\beta$ yields (3)$_\beta$ and (5)$_\beta$, (7)$_\beta$, (9)$_\beta$ are irrelevant for a successor $\beta$. 
\end{proof}

Suppose $\langle M_\alpha,a_\alpha:\alpha<\alpha^* \rangle ^\frown \langle M_{\alpha^*} \rangle$ is independent and $M_0$ is of cardinality $\lambda$. Suppose that $M_0 \preceq N_0 \in K_{\lambda^+}$. Proposition \ref{widehat{NF} with independence} says that in such a case, we can amalgamate $M_{\alpha^*}$ and $N_0$ over $M_0$ such that the sequence $\langle a_\alpha:\alpha<\alpha^* \rangle$ will be independent. Moreover, we can choose the amalgamation such that $\widehat{NF}$ will hold.
\begin{proposition} \label{widehat{NF} with independence}
If 
\begin{enumerate}
\item $\alpha^*<\lambda^+$,
\item $M_\alpha \in K_\lambda$ for every $\alpha \leq \alpha^*$, \item the sequence $\langle M_\alpha,a_\alpha:\alpha<\alpha^* \rangle ^\frown \langle M_{\alpha^*} \rangle$ is independent and \item $M_0 \preceq N_0 \in K_{\lambda^+}$
\end{enumerate}
 then for some $N_1 \in K_{\lambda^+}$ and some embedding $f:N_0 \to N_1$ fixing $M_0$ pointwise, the sequence $\langle a_\alpha:\alpha<\alpha^* \rangle$ is independent in $(f[N_0],N_1)$ and  $\widehat{NF}(M_0,M_{\alpha^*},f[N_0],N_1)$ holds. 
\begin{displaymath}
\xymatrix{M_{\alpha^*} \ar[r] & N_1 \\
 M_\alpha \ar[u]^{id} \\
M_0 \ar[u]^{id} \ar[r]_{id} & N_0 \ar[uu]_{f}
}
\end{displaymath}
\end{proposition}

\begin{proof}
Let $\langle M^b_\beta:\beta<\lambda^+ \rangle$ be a filtration of $N_0$ such that $M^b_0=M_0$. For every $\alpha<\alpha^*$ define $M^a_\alpha=:M_\alpha$. By Proposition \ref{a rectangle with NF}, there is a `rectangle of models' $\{M_{\alpha,\beta}:\alpha<\alpha^*, \beta<\lambda^+ \}$ and a set of embeddings $\{f_\beta:\beta<\lambda^+\}$ satisfiying Clauses (1)-(10) of Proposition \ref{a rectangle with NF}. Define $M_{\alpha,\lambda^+}=:\bigcup_{\beta<\lambda^+}M_{\alpha,\beta}$ for every $\alpha<\alpha^*$. 

\begin{displaymath}
\xymatrix{M_{\alpha^*} \ar[r]^{id} & M_{\alpha^*,\beta} \ar[r]^{id} & M_{\alpha^*,\beta+1} \ar[r]^{id} & N_1\\
a_\alpha \in M_{\alpha+1} \ar[u]^{id} \ar[r]^{id} & M_{\alpha+1,\beta} \ar[r]^{id}  \ar[u]^{id} & M_{\alpha+1,\beta+1} \ar[r]^{id} \ar[u]^{id} & M_{\alpha+1,\lambda^+} \ar[u]^{id} \ni a_\alpha \\
 M_\alpha \ar[u]^{id} \ar[r]^{id} & M_{\alpha,\beta} \ar[r]^{id} \ar[u]^{id} & M_{\alpha,\beta+1} \ar[r]^{id}  \ar[u]^{id} & M_{\alpha,\lambda^+}  \ar[u]^{id} \\
M_0 \ar[u]^{id} \ar[r]_{id} & M_{0,\beta} \ar[r]^{id} \ar[u]^{id}  & M_{0,\beta+1} \ar[u]^{id} \ar[r]^{id} & M_{0,\lambda^+} \ar[u]^{id} \\
M^b_0 \ar[u]^{=} \ar[r]_{id} & M^b_{\beta} \ar[u]^{f_{\beta}} \ar[r]_{id} & M^b_{\beta+1} \ar[u]^{f_{\beta+1}} \ar[r]_{id} & N_0 \ar[u]^{f} }
\end{displaymath}

$\langle M_{\alpha,\lambda^+}:\alpha \leq \alpha^* \rangle$ is an increasing continuous sequence of models of cardinality $\lambda^+$. By Clauses (3),(5) and (10) of Proposition \ref{a rectangle with NF}, we have $\widehat{NF}(M_\alpha,M_{\alpha+1},M_{\alpha,\lambda^+},M_{\alpha+1,\lambda^+})$. But the relation $\widehat{NF}$ respects the frame and $tp(a_\alpha,M_\alpha,M_{\alpha+1})$ is basic. So $tp(a_\alpha,M_{\alpha,\lambda^+},M_{\alpha+1,\lambda^+})$ does not fork over $M_\alpha$. By the definition of independence, $tp(a_\alpha,M_\alpha,M_{\alpha+1})$ does not fork over $M_0$. By transitivity, $tp(a_\alpha,M_{\alpha,\lambda^+},M_{\alpha+1,\lambda^+})$ does not fork over $M_0$. Therefore by monotonicity, it does not fork over $M_{0,\lambda^+}$. Hence, the sequence $\langle M_{\alpha,\lambda^+},a_\alpha:\alpha<\alpha^* \rangle ^\frown \langle M_{\alpha^*,\lambda^+} \rangle$ is independent.

Define $N_1:=\bigcup_{\alpha<\alpha^*}M_{\alpha,\lambda^+}$ and $f:=\bigcup_{\beta<\lambda^+}f_\beta$. 

It remains to prove that $\widehat{NF}(M_0,M_{\alpha^*},f[N_0],N_1)$ holds. It can be proven, by using the long transitivity of the relation $\widehat{NF}$, but we prefer to prove it in another way. Define $M_{\alpha^*,\beta}:=\bigcup_{\alpha<\alpha^*}M_{\alpha,\beta}$ for every $\beta<\lambda^+$. By Clauses (2),(4) and (10) of Proposition \ref{a rectangle with NF}, we have $NF(M_{0,\beta},M_{0,\beta+1},M_{\alpha^*,\beta},M_{\alpha^*,\beta+1})$ for each $\beta<\lambda^+$, or equivalently, $NF(f_\beta[M^b_{\beta}],f_{\beta+1}[M^b_{\beta+1}],M_{\alpha^*,\beta},M_{\alpha^*,\beta+1})$. Since the sequences $\langle f_\beta[M^b_\beta]:\beta<\lambda^+ \rangle$ and $\langle M_{\alpha^*,\beta}:\beta<\lambda^+ \rangle$ are increasing and continuous, $NF(M_0,M_{\alpha^*},f[N_0],N_1)$ holds.
\end{proof}
 
The following proposition says that the relation $\widehat{NF}$ respects independence.  
\begin{proposition}\label{widehat{NF} implies independence}\label{widehat{NF} respects independence}
Let $\alpha^*<\lambda^+$. Let $M_0^-$ and $M_1^-$ be models of cardinality $\lambda$ and let $M_0$ and $M_1$ be models of cardinality $\lambda^+$ satisfying $\widehat{NF}(M_0^-,M_1^-,M_0,M_1)$. Then every independent sequence, $\langle a_\alpha:\alpha<\alpha^* \rangle$, in $(M_0^-,M_1^-)$ is independent in $(M_0,M_1)$ as well.
\end{proposition}

\begin{proof}
By Proposition \ref{widehat{NF} with independence}, we can find $N_1 \in K_{\lambda^+}$ and an embedding $f:M_0 \to N_1$ such that $\widehat{NF}(M_0^-,M_1^-,f[M_0],N_1)$ holds and the sequence $\langle a_\alpha:\alpha<\alpha^* \rangle$ is independent in $(f[M_0],N_1)$. 

By weak uniqueness of $\widehat{NF}$ (Fact \ref{the widehat{NF}-properties}(d)), there is a model $N_2 \in K_{\lambda^+}$ and an embedding $g:N_1 \to N_2$ such that the following diagram commutes:
\begin{displaymath}
\xymatrix{& N_1 \ar[rr]^{g} && N_2 \\
M_0 \ar[ru]^{f} \ar[rr]^{id} && M_1 \ar[ru]^{id} \\
M_0^- \ar[u]^{id} \ar[r]^{id} & M_1^- \ar[uu]^{id} \ar[ru]^{id}
}
\end{displaymath}

So $g(f(x))=x$, for every $x \in M_0$. Hence, the sequence is independent in $(M_0,N_2)$, or equivalently in $(M_0,M_1)$.  
\end{proof}

The following theorem already generalized by Boney and Vasey \cite[Corollary 4.10 and Remark 4.11]{bo6va4}: they eliminated Hypotheses (2) and (3).
\begin{theorem}\label{continuity of serial independence}\label{if the class of uniqueness triples satisfies the existence property then continuity}
Suppose:
\begin{enumerate}
\item $\frak{s}$ is a semi-good non-forking $\lambda$-frame \item $\frak{s}$ satisfies the conjugation property, \item $K^{3,uq}$ satisfies the existence property, \item $(K,\preceq)$ satisfies the amalgamation in $\lambda^+$ property and \item $(K,\preceq)$ satisfies the $(\lambda,\lambda^+)$-tameness for non-forking types property.
\end{enumerate}

Then the $(\lambda,\lambda^+)$-continuity of serial independence property holds.
\end{theorem}

\begin{proof}
Let $\beta^*<\lambda^+$, $M_1 \in K_{\lambda^+}$, $M_1 \preceq M_2 \in K_{\lambda^+}$ and let $\langle M_{1,\alpha}:\alpha<\lambda^+ \rangle$ be a filtration of $M_1$. Suppose $\langle a_\beta:\beta<\beta^* \rangle$ is independent in $(M_{1,\alpha},M_2)$ for each $\alpha<\lambda^+$. We have to prove that $\langle a_\beta:\beta<\beta^* \rangle$ is independent in $(M_1,M_2)$.

By Theorem \ref{we can use NF}, $M_1 \preceq^{NF}_{\lambda^+}M_2$. Let $\langle M_{2,\alpha}:\alpha<\lambda^+ \rangle$ be a filtration of $M_2$. By Remark \ref{remark preceq^{NF}}, there is a club $E$ of $\lambda^+$ such that for every $\alpha \in E$, $\widehat{NF}(M_{1,\alpha},M_{2,\alpha},M_1,M_2)$.   Define $J=:\{a_\beta:\beta<\beta^*\}$. $J \subseteq M_2$. Since $|J|<\lambda^+$, for some $\alpha \in E$ we have $J \subseteq M_{2,\alpha}$. But 
%by Fact \ref{the widehat{NF}-properties}, 
$\widehat{NF}(M_{1,\alpha},N_{2,\alpha},M_1,M_2)$. So by Proposition \ref{widehat{NF} implies independence}, $\langle a_\beta:\beta<\beta^* \rangle$ is independent in $(M_1,M_2)$.
\end{proof}

Theorem \ref{continuity of serial independence for saturated models} is the analog of Theorem \ref{continuity of serial independence}, where we restrict ourselves to the saturated models.
\begin{theorem}\label{continuity of serial independence for saturated models}
Suppose:
\begin{enumerate}
\item $\frak{s}$ is a semi-good non-forking $\lambda$-frame \item $\frak{s}$ satisfies the conjugation property, \item $K^{3,uq}$ satisfies the existence property, \item Every saturated model in $\lambda^+$ over $\lambda$ is an amalgamation base and \item $(K,\preceq)$ satisfies the $(\lambda,\lambda^+)$-tameness for non-forking types over saturated models property.
\end{enumerate}

Then the $(\lambda,\lambda^+)$-continuity of serial independence for saturated models property holds.
\end{theorem}

\begin{proof}
Similar to the proof of Theorem \ref{continuity of serial independence}, but here we apply Theorem \ref{we can use NF with tameness for saturated models} in place of Theorem \ref{we can use NF}.
\end{proof}

%%%%%%%%%%%%%%%%%%%%%%%%
%   section
%%%%%%%%%%%%%%%%%%%%%%%%
\section{Getting Good Non-Forking $\lambda^+$-Frames}
Recall the definition of the third candidate of a good non-forking $\lambda^+$-frame (of the second version).
\begin{definition}
$$\frak{s^+}=:(K^{sat},\preceq^{NF}_{\lambda^+} \restriction K^{sat},S^{bs}_{\lambda^+} \restriction K^{sat},\dnf \restriction (K^{sat},\preceq^{NF}_{\lambda^+})),$$ 
where $\dnf \restriction (K^{sat},\preceq^{NF}_{\lambda^+})$ is the set of quadruples, $(M_0,M_1,a,M_2)$ in $\dnf_{\lambda^+}$ such that $M_0,M_1,M_2 \in K^{sat}$, $M_0 \preceq^{NF}_{\lambda^+}$ and $M_1 \preceq^{NF}_{\lambda^+} M_2$.
\end{definition}

In this section, we present sufficient conditions for $\frak{s}_{\lambda^+}$ and $\frak{s^+}$ being good non-forking $\lambda^+$-frames.

In Corollary \ref{corollary a good non-forking frame}, we present sufficient conditions for $\frak{s}_{\lambda^+}$ being a good non-forking $\lambda^+$-frame. Boney and Vasey  \cite[Theorem 2.1]{bo6va4} have already generalized Corollary \ref{corollary a good non-forking frame}, eliminating two hypotheses:
\begin{enumerate}
\item $\frak{s}$ satisfies the conjugation property and \item the class of uniqueness triples satisfies the existence property.
\end{enumerate}

\begin{corollary}\label{corollary a good non-forking frame}
Suppose:
\begin{enumerate}
\item $(K,\preceq)$ satisfies the joint embedding and amalgamation properties in $\lambda$ and in $\lambda^+$,
\item $\frak{s}=(K,\preceq,S^{bs},\dnf)$ is a semi-good non-forking $\lambda$-frame,
\item $\frak{s}$ satisfies the conjugation property,  
\item $(K,\preceq)$ satisfies $(\lambda,\lambda^+)$-tameness for non-forking types and 
\item the class of uniqueness triples satisfies the existence property.
\end{enumerate}
Then $\frak{s}_{\lambda^+}$ is a good non-forking $\lambda^+$-frame.
\end{corollary}

\begin{proof}
By Theorem \ref{if the class of uniqueness triples satisfies the existence property then continuity}, the $(\lambda,\lambda^+)$-continuity of serial independence property holds. Hence, by Theorem \ref{continuity implies good frame} and Condition (1) (here), $\frak{s}_{\lambda^+}$ is a good non-forking $\lambda^+$-frame.
\end{proof}

\begin{remark}
We skip two additional proofs of Corollary \ref{corollary a good non-forking frame}:

(a) Since the relations $\preceq \restriction K_{\lambda^+}$ and $\preceq ^{NF}_{\lambda^+}$ are equivalent and the relation $\preceq \restriction K_{\lambda^+}$ satisfies smoothness, the relation $\preceq^{NF}_{\lambda^+}$ satisfies smoothness. So by the proof of \cite[Theorem 10.1.9]{jrsh875}, $\frak{s}_{\lambda^+}$ is a good non-forking $\lambda^+$-frame (\cite[Theorem 10.1.9]{jrsh875} relates to $K^{sat}$, but we use its proof only).

(b) Since the relations $\preceq \restriction K_{\lambda^+}$ and $\preceq ^{NF}_{\lambda^+}$ are equivalent, by the proof of \cite[Proposition 10.1.12]{jrsh875}, $\frak{s}_{\lambda^+}$ satisfies the symmetry axiom. Hence, by Theorem \ref{symmetry implies good non-forking frame}, $\frak{s}_{\lambda^+}$ is a good non-forking $\lambda^+$-frame. 
\end{remark}

In order to apply the methods of \cite{shh}.III, we should study $\frak{s^+}$, not $\frak{s}_{\lambda^+}$! In Corollary \ref{corollary a good non-forking frame assuming tameness for non-forking types over saturated models}, we present sufficient conditions for $\frak{s^+}$ being a good non-forking $\lambda^+$-frame.

While Hypotheses (1) and (4) of Corollary \ref{corollary a good non-forking frame} relate to all the models of cardinality $\lambda^+$, in Theorem \ref{corollary a good non-forking frame assuming tameness for non-forking types over saturated models}, they relate to the models in $K^{sat}$ only!

Let us compare the hypotheses of Theorem \ref{corollary a good non-forking frame assuming tameness for non-forking types over saturated models} with the hypotheses in \cite[Theorem 11.1.5]{jrsh875}. On the one hand, in \cite[Theorem 11.1.5]{jrsh875} we assume that there are not many models of cardinality $\lambda^{++}$ (in addition to set theoretical hypotheses), in order to prove (the existence of uniqueness triples and) that the relation $\preceq^{NF}_{\lambda^+}$ satisfies smoothness. Here, we (assume that the class of uniqueness triples satisfies the existence property and) get smoothness by the equivalence between the relations. So we do not have to assume that there are not many models of cardinality $\lambda^{++}$. On the other hand, here, we assume that every saturated model in $\lambda^+$ over $\lambda$ is an amalgamation base and that the $(\lambda,\lambda^+)$-tameness for non-forking types over saturated models property holds. 

Now we prove Theorem \ref{1}:
\begin{theorem}\label{corollary a good non-forking frame assuming tameness for non-forking types over saturated models}
Suppose:
\begin{enumerate}
\item $(K,\preceq)$ satisfies the joint embedding and amalgamation properties in $\lambda$ (actually, implied by (2)),
\item $\frak{s}=(K,\preceq,S^{bs},\dnf)$ is a semi-good non-forking $\lambda$-frame,
\item $\frak{s}$ satisfies the conjugation property,  
\item every saturated model in $\lambda^+$ over $\lambda$ is an amalgamation base,
\item $(K,\preceq)$ satisfies the $(\lambda,\lambda^+)$-tameness for non-forking types over saturated models property and 
\item the class of uniqueness triples satisfies the existence property.
\end{enumerate}
Then $\frak{s^+}$ is a good non-forking $\lambda^+$-frame.
\end{theorem}

We exhibit three proofs:

\begin{proof}
By Theorem \ref{continuity of serial independence for saturated models}, the $(\lambda,\lambda^+)$-continuity of serial $K^{sat}$-independence property holds. By Fact \ref{preceq^{NF}-properties}(d), $(K^{sat},\preceq^{NF}_{\lambda^+} \restriction K^{sat})$ is closed under unions of increasing continuous sequences of cardinality less than $\lambda^{++}$. By Theorem \ref{the main theorem of the paper}, the relations $\preceq \restriction K^{sat}$ and $\preceq^{NF}_{\lambda}$ are equivalent. So $(K^{sat},\preceq \restriction K^{sat})$ is closed under unions of increasing continuous sequences of cardinality less than $\lambda^{++}$. Therefore by Remark \ref{K sat is AEC iff}, $(K^{sat},\preceq \restriction K^{sat})$ is an AEC in $\lambda^+$. So by Theorem \ref{continuity implies good frame for sat}, $\frak{s}^{sat}$ is a good non-forking $\lambda^+$-frame minus amalgamation in $\lambda^+$. But by Theorem \ref{the main theorem of the paper}, $\frak{s}^{sat}$ is $\frak{s}^+$. By Clause (4), $\frak{s}^+$ satisfies the amalgamation property in $\lambda^+$. 
\end{proof}

\begin{proof}
By Theorem \ref{we can use NF with tameness for saturated models}, the relations $\preceq \restriction K^{sat}$ and $\preceq^{NF}_{\lambda^+} \restriction K^{sat}$ are equivalent. Since the relation $\preceq \restriction K^{sat}$ satisfies smoothness (because $(K,\preceq)$ is an AEC), so too the relation $\preceq^{NF}_{\lambda^+} \restriction K^{sat}$. So by \cite[Theorem 10.1.9]{jrsh875}, $\frak{s^+}$ is a good non-forking $\lambda^+$-frame.  
\end{proof}

\begin{proof}
Since the relations $\preceq \restriction K^{sat}$ and $\preceq^{NF}_{\lambda^+} \restriction K^{sat}$ are 
equivalent, by the proof of \cite[Proposition 10.1.12]{jrsh875}, $\frak{s}^{sat}$ (namely, $\frak{s^+}$) satisfies the symmetry axiom. Hence, by Theorem \ref{symmetry implies good-frame for sat}, $\frak{s^+}$ is a good non-forking $\lambda^+$-frame. 
\end{proof}

%%%%%%%%%%%%%%%%%%%%%%
%%    section     %%%%%%%%%%%%%%
%%%%%%%%%%%%%%%%%%%%%%
\section{Meanings of Galois-Types}
In this section, we show that even if the relations $\preceq^{NF}_{\lambda^+}$ and $\preceq$ are not equivalent, the definitions of `a type' in the different AECs considered here, coincide.

Let $M_0,M_1$ be two models in $K_{\lambda^+}$. Suppose that $M_0 \preceq^{NF}_{\lambda^+} M_1$ and $a$ is an element in $M_1-M_0$. Naturally, we have in mind two meanings of `the galois-type of $a$'. On the one hand, we can consider the galois-type of $a$ in $M_1$ over $M_0$ in the context of $(K,\preceq)$. On the other hand, we can consider the galois-type of $a$ in $M_1$ over $M_0$ in the context of $(K,\preceq^{NF}_{\lambda^+})$. Moreover, if we consider restricting ourselves to the saturated models in $\lambda^+$ over $\lambda$, we get two additional meanings of `the type of $a$'.

Formally, a galois-type is defined (at the beginning of Section 3) as an equivalence class of a relation, $E$. $E$ is defined as the transitive closure of a relation, $E^*$. But $E^*$ depends on the class of models and the relation between the models.

Theorem \ref{the notions of type coincide} says that in four specific pairs, $(K',\preceq')$, the definition of $E^*$ coincides. So the four meanings of `the galois-type of $a$' coincide. 
 
 \cite[Proposition 10.1.4]{jrsh875} is similar to Theorem \ref{the notions of type coincide} but the proof of \cite[Proposition 10.1.4]{jrsh875} is wrong. %Note that in Proposition \ref{the notions of type coincide}, we do not assume that the relations $\preceq \restriction K_{\lambda^+}$ and $\preceq^{NF}_{\lambda^+}$ are equivalent.

\begin{definition}\label{the definition of E^*}
Let $K' \subseteq K$ and let $\preceq'$ be a binary relation on $K'$ (or on $K$) such that if $M_0 \preceq' M_1$ then $M_0 \preceq M_1$. We define $K^{3,K',\preceq'}:=\{(M_0,M_1,a):M_0,M_1 \in K',\ M_0 \preceq' M_1$ and $a \in M_1 -M_0\}$. 

We define a binary relation $E^{*,K',\preceq'}$ on the class of triples $K^{3,K',\preceq'}$ by: $(M_0,M_1,a_1)E^{*,K',\preceq'}(M_0,M_2,a_2)$ (in words `$(M_0,M_1,a_1)$ and $(M_0,M_2,a_2)$ are of the same galois-type in $(K',\preceq')$') if and only if there is an amalgamation $(f_1,f_2,M_3)$ of $M_1$ and $M_2$ over $M_0$ such that $M_3$ is in $K'$, $f_1[M_1] \preceq' M_3$, $f_2[M_2] \preceq' M_3$ and $f_1(a_1)=f_2(a_2)$. 
\end{definition}

\begin{example}
The following examples are of interest for us: 
\begin{enumerate}
\item $K^{3,K_{\lambda^+},\preceq }=\{(M_0,M_1,a):M_0,M_1 \in K_{\lambda^+},\ M_0 \preceq M_1,\ M_0 \preceq M_2$ and $a \in M_1-M_0\}$, 
\item $K^{3,K^{sat},\preceq}=\{(M_0,M_1,a) \in K^{3}_{\lambda^+}:M_0,M_1 \in K^{sat}\}$, 
\item $K^{3,K_{\lambda^+},\preceq^{NF}_{\lambda^+}}=\{(M_0,M_1,a) \in K^{3}_{\lambda^+}:M_0 \preceq^{NF}_{\lambda^+} M_1$ and $M_0 \preceq^{NF}_{\lambda^+} M_1 \}$, 
\item $K^{3,K^{sat},\preceq^{NF}_{\lambda^+}}=K^{3,K^{sat},\preceq} \cap K^{3,K_{\lambda^+},\preceq^{NF}_{\lambda^+}}$.
\end{enumerate}
\end{example}

\begin{proposition}\label{a completion for section 2}
Let $M_0,M_1$ and $M_2$ be three saturated models in $\lambda^+$ over $\lambda$. Let $a_1 \in M_1-M_0$ and let $a_2 \in M_2-M_0$. Then $(M_0,M_1,a_1)E^{*,K,\preceq} (M_0,M_2,a_2)$ holds if and only if $(M_0,M_1,a_1)E^{*,K^{sat},\preceq} (M_0,M_2,a_2)$ holds. 
\end{proposition}

\begin{proof}
If $(M_0,M_1,a_1)E^{*,K^{sat},\preceq} (M_0,M_2,a_2)$ then by definition $(M_0,M_1,a_1) \allowbreak E^{*,K,\preceq} (M_0,M_2,a_2)$. Conversely, suppose $(M_0,M_1,a_1)E^{*,K,\preceq} (M_0,M_2,a_2)$. So is an amalgamation $(f_1,f_2,M_3)$ of $M_1$ and $M_2$ over $M_0$ such that $f_1(a_1)=f_2(a_2)$. By Proposition \ref{extending to K sat}, there is a model $M_4 \in K^{sat}$ with $M_3 \preceq M_4$. So the amalgamation $(f_1,f_2,M_4)$ is as needed.   
\end{proof}

We restate \cite[Proposition 7.1.13.b]{jrsh875} as follows:
\begin{fact}\label{the version of tameness in 875}
Suppose:
\begin{enumerate}
\item $M_0 \prec^+_{\lambda^+} M_1$, \item $M_0 \prec^+_{\lambda^+} M_2$, \item $a_1 \in M_1-M_0$, \item $a_2 \in M_2-M_0$, \item for each $N \in K_\lambda$ with $N \preceq M_0$, we have $tp(a_1,N,M_1)=tp(a_2,N,M_2)$. 
\end{enumerate}
Then there is an isomorphism $f:M_1 \to M_2$ fixing $M_0$ pointwise, such that $f_1(a_1)=a_2$.
\end{fact}
[A note: \cite[Proposition 7.1.17.b]{jrsh875} is a wrong application of \cite[Proposition 7.1.13.b]{jrsh875}. In order to fix it, we should add to its assumptions the assumption that $M_0 \preceq^{NF}_{\lambda^+} M_1$ and $M_0 \preceq^{NF}_{\lambda^+} M_2$].

\begin{proposition}\label{preparation for `the notions of type coincide'}
Suppose $(M_0,M_1,a_1) E^{*,K_{\lambda^+},\preceq} (M_0,M_2,a_2)$. Then we can find an amalgamation $(f_1,f_2,M_3)$ such that the following hold:
\begin{enumerate} 
\item  $M_3 \in K^{sat}$, \item $f_1(a_1)=f_2(a_2)$ and \item if $M_0 \preceq^{NF}_{\lambda^+} M_1$ and $M_0 \preceq^{NF}_{\lambda^+} M_2$ then $f_1[M_1] \preceq^{NF}_{\lambda^+} M_3$ and $f_2[M_2] \preceq^{NF}_{\lambda^+} M_3$.
\end{enumerate} 
\end{proposition}

\begin{proof}
By the definition of $E^{*,K_{\lambda^+},\preceq}$, there is an amalgamation $(f_1,f_2,M_3)$ of $M_1$ and $M_2$ over $M_0$ such that $f_1(a_1)=f_2(a_2)$. If $M_0 \npreceq^{NF}_{\lambda^+} M_1$ or $M_0 \npreceq^{NF}_{\lambda^+} M_2$ then Condition (3) is irrelevant. In this case, we apply Proposition \ref{extending to K sat}, replacing $M_3$ by a model $M_3^+$ such that $M_3 \preceq M_3^+$ and $M_3^+ \in K^{sat}$.

Suppose $M_0 \preceq^{NF}_{\lambda^+} M_1$ and $M_0 \preceq^{NF}_{\lambda^+} M_2$.  By Fact \ref{prec^+-properties}(1), we can take $M_1^+,M_2^+$ such that $f_1[M_1] \prec^+ M_1^+$ and $f_2[M_2] \prec^+ M_2^+$. 
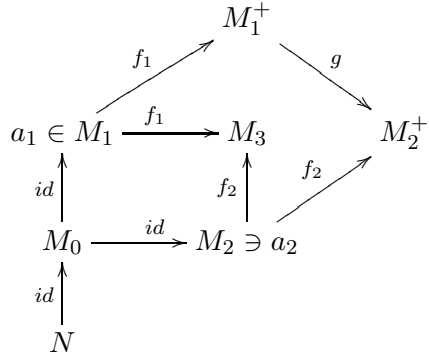
\begin{displaymath}
\xymatrix{& M_1^+ \ar[rd]^{g} \\
a_1 \in M_1 \ar[r]^{f_1} \ar[ru]^{f_1} & M_3 & M_2^+  \\
M_0 \ar[r]^{id} \ar[u]^{id} & M_2 \ni a_2 \ar[ru]^{f_2} \ar[u]^{f_2} \\
N \ar[u]^{id} 
}
\end{displaymath}

For every $N \in K_\lambda$ with $N \preceq M_0$, we have $tp(f_1(a_1),N,M_1^+)=tp(f_1(a_1),\allowbreak N,f_1[M_1])=tp(f_1(a_1),N,M_3)=tp(f_2(a_2),N,M_3)=tp(f_2(a_2),N,f_2[M_2])\allowbreak =tp(f_2(a_2),N \allowbreak ,M_2^+)$. So by Fact \ref{the version of tameness in 875}, we can find an isomorphism $g:M_1^+ \to M_2^+$ fixing $M_0$ pointwise such that $g(f_1(a_1))=f_2(a_2)$. Now \linebreak $(g \circ f_1,f_2,M_2^+)$ is an amalgamation of $M_1$ and $M_2$ over $M_0$ as needed (by Fact \ref{prec^+-properties}(2), $g \circ f_1[M_1] \preceq^{NF}_{\lambda^+} g[M_1^+]=M_2^+$, $f_2[M_2] \preceq^{NF}_{\lambda^+} M_2^+$ and $M_2^+ \in K^{sat}$).
\end{proof}

By the following theorem, $E^{*,K^{sat},\preceq},E^{*,K_{\lambda^+},\preceq^{NF}_{\lambda^+}}$ and $E^{*,K^{sat},\preceq^{NF}_{\lambda^+}}$ are the restrictions of $E^{*,K,\preceq}$ to $K^{3,K^{sat},\preceq},K^{3,K_{\lambda^+},\preceq^{NF}_{\lambda^+}}$ and $K^{3,K^{sat},\preceq^{NF}_{\lambda^+}}$, respectively.

\begin{theorem}\label{the notions of type coincide}
The definitions of `galois-type' in the following contexts coincide: $\{(K_{\lambda^+},\preceq),(K_{\lambda^+} \allowbreak ,\preceq^{NF}_{\lambda^+}),(K^{sat},\preceq),(K^{sat},\preceq^{NF}_{\lambda^+})\}$.

More precisely: Let $(K',\preceq')$ and $(K'',\preceq'')$ be two pairs in $\{(K_{\lambda^+},\preceq),(K_{\lambda^+} \allowbreak ,\preceq^{NF}_{\lambda^+}),(K^{sat},\preceq),(K^{sat},\preceq^{NF}_{\lambda^+})\}$. Let $(M_0,M_1,a_1),(M_0,M_2,a_2)$ be two triples in $K^{3,K',\preceq'} \cap K^{3,K'',\preceq''}$. Then $$(M_0,M_1,a_1) E^{3,K',\preceq'}(M_0,M_2,a_2) \Leftrightarrow (M_0,M_1,a_1) E^{3,K'',\preceq''}(M_0,M_2,a_2).$$
\end{theorem}

\begin{proof}
By Proposition \ref{preparation for `the notions of type coincide'}.
\end{proof}

Note that in the proof of \cite[Conclusion II.8.7]{shh}, the equivalence between the meanings of galois-types was not used (we quote: `the equality of types there is a formal, not the true one'). Theorem \ref{the notions of type coincide} gives a way to simplify the proof of \cite[Conclusion II.8.7]{shh}. But since it is not a new result, we will not elaborate.

\section{Continuations}
Several results and ideas that appear in the current paper, were already applied. In \cite{jrprime}, we apply Theorem \ref{0}, proving that the class of primeness triples satisfies the existence property. We should study connections between the results here and the results of \cite[III]{shh}. 
%I am alomost sure that Boney and Vasey \cite{complete...} applies \ref{continuity implies symmetry} in their proof of symmetry complete...
Vasey \cite{vas6} continues Theorem \ref{we can use NF} \cite[Fact 11.15]{vas6}, doing an important step toward a proof of the categoricity conjecture.

\subsection*{Acknowledgments}

First of all, I would like to thank God.

I wish to thank my wife for her
support.

%I wish to acknowledge Prof John T.Baldwin for doing an outstanding work on improving the paper.

I wish to acknowledge wholeheartedly Prof David Webb for doing an outstanding work on improving the paper.
 
I would like to thank Efrat Taub, Sebastien Vasey, Will Boney and Prof Rami Grossberg for their comments and suggestions.

\end{document}